\documentclass[11pt]{amsart}
\usepackage{tikz,amsthm,amsmath,amstext,amssymb,amscd,epsfig,euscript, mathrsfs, dsfont,pspicture,multicol,graphpap,graphics,graphicx,times,enumerate,subfig,sidecap,wrapfig,color}
 \usepackage [latin1]{inputenc}

\usepackage{amsmath}
\usepackage{amssymb}
\usepackage{pgf}
\usepackage{stackrel}
\usepackage[colorlinks,citecolor=red]{hyperref}

\newcommand{\R}{{\mathbb R}}       
\newcommand{\DD}{{\mathcal D}}

\newcommand{\HH}{{\mathcal H}}

\newcommand{\RR}{{\mathcal R}}

\newcommand{\EE}{{\mathcal E}}

\newcommand{\om}{{\Omega}}
\newcommand{\hm}{{\omega}}

\newcommand{\diam}{\mathop{\rm diam}}
\newcommand{\dist}{{\rm dist}}

\newcommand{\rf}[1]{{(\ref{#1})}}

\newcommand{\supp}{\operatorname{supp}}

\newcommand{\vphi}{{\varphi}}
\newcommand{\ve}{{\varepsilon}}
\newcommand{\vv}{{\vspace{2mm}}}
\newcommand{\vvv}{\vspace{4mm}}
\newcommand{\wt}[1]{{\widetilde{#1}}}

\newcommand{\noi}{\noindent}

\newcommand{\LD}{{\mathsf{LD}}}
\newcommand{\LLD}{{\mathcal{LD}}}

\def\Xint#1{\mathchoice
{\XXint\displaystyle\textstyle{#1}}%
{\XXint\textstyle\scriptstyle{#1}}%
{\XXint\scriptstyle\scriptscriptstyle{#1}}%
{\XXint\scriptscriptstyle\scriptscriptstyle{#1}}%
\!\int}
\def\XXint#1#2#3{{\setbox0=\hbox{$#1{#2#3}{\int}$ }
\vcenter{\hbox{$#2#3$ }}\kern-.58\wd0}}

\def\avint{\Xint-}

\newtheorem{theorem}{Theorem}[section]
\newtheorem{lemma}[theorem]{Lemma}

\newtheorem{proposition}[theorem]{Proposition}

\newtheorem*{lemma*}{Lemma}
\newtheorem*{theorem*}{Theorem}

\theoremstyle{definition}
\newtheorem{definition}[theorem]{Definition}

\theoremstyle{remark}
\newtheorem{rem}[theorem]{\bf Remark}

\numberwithin{equation}{section}

\newcommand{\cnj}[1]{\overline{#1}}
\newcommand{\RRem}{\begin{rem}}
\newcommand{\erem}{\end{rem}}

\def\d{\partial}

\makeatletter
\def\@tocline#1#2#3#4#5#6#7{\relax
  \ifnum #1>\c@tocdepth 
  \else
    \par \addpenalty\@secpenalty\addvspace{#2}%
    \begingroup \hyphenpenalty\@M
    \@ifempty{#4}{%
      \@tempdima\csname r@tocindent\number#1\endcsname\relax
    }{%
      \@tempdima#4\relax
    }%
    \parindent\z@ \leftskip#3\relax \advance\leftskip\@tempdima\relax
    \rightskip\@pnumwidth plus4em \parfillskip-\@pnumwidth
    #5\leavevmode\hskip-\@tempdima
      \ifcase #1
       \or\or \hskip 1em \or \hskip 2em \else \hskip 3em \fi%
      #6\nobreak\relax
    \dotfill\hbox to\@pnumwidth{\@tocpagenum{#7}}\par
    \nobreak
    \endgroup
  \fi}
\makeatother

\def\cD{{\mathscr{D}}}

\def\cH{{\mathcal{H}}}

\def\cL{{\mathscr{L}}}

\def\bB{{\mathbb{B}}}

\def\bR{{\mathbb{R}}}

\def\bZ{{\mathbb{Z}}}

\newcommand{\ps}[1]{\left( #1 \right)}

\def\gec{\gtrsim}
\def\lec{\lesssim}

\def\grad{\nabla}

\def\bR{\mathbb{R}}
\def\lec{\lesssim}

\def\warrow{\rightharpoonup}

\begin{document}

\title[A two-phase free boundary problem for harmonic measure]{A  two-phase free boundary problem for harmonic measure and uniform rectifiability}


\author[Azzam]{Jonas Azzam}

\address{Jonas Azzam\\
School of Mathematics \\ University of Edinburgh \\ JCMB, Kings Buildings \\
Mayfield Road, Edinburgh,
EH9 3JZ, Scotland.}
\email{j.azzam "at" ed.ac.uk}

\newcommand{\jonas}[1]{\marginpar{\color{magenta} \scriptsize \textbf{Jonas:} #1}}

\author[Mourgoglou]{Mihalis Mourgoglou}

\address{Mihalis Mourgoglou\\
Departamento de Matem\'aticas, Universidad del Pa\' is Vasco, Barrio Sarriena s/n 48940 Leioa, Spain and\\
Ikerbasque, Basque Foundation for Science, Bilbao, Spain.
}
\email{michail.mourgoglou@ehu.eus}

\author[Tolsa]{Xavier Tolsa}
\address{Xavier Tolsa
\\
ICREA, Passeig Lluís Companys 23 08010 Barcelona, Catalonia, and\\
Departament de Matem\`atiques and BGSMath
\\
Universitat Aut\`onoma de Barcelona
\\
Edifici C Facultat de Ci\`encies
\\
08193 Bellaterra (Barcelona), Catalonia
}
\email{xtolsa@mat.uab.cat}

\thanks{M.M. was supported  by IKERBASQUE and partially supported by the grant MTM-2017-82160-C2-2-P of the Ministerio de Econom\'ia y Competitividad (Spain), and by  IT-1247-19 (Basque Government). X.T. was supported by the ERC grant 320501 of the European Research Council (FP7/2007-2013) and partially supported by MTM-2016-77635-P,  MDM-2014-044 (MICINN, Spain), 2017-SGR-395 (Catalonia), and by Marie Curie ITN MAnET (FP7-607647).
}

\begin{abstract}
We assume that $\om_1, \om_2 \subset \R^{n+1}$, $n \geq 1$ are two disjoint domains whose complements satisfy { the capacity density} condition and the intersection of their boundaries $F$ has positive harmonic measure. Then we show that in a fixed ball $B$ centered on $F$, if the harmonic measure of $\om_1$ satisfies a scale invariant $A_\infty$-type condition with respect to the harmonic measure of $\om_2$ in $B$, then there exists a uniformly $n$-rectifiable set $\Sigma$  so that the harmonic measure of $\Sigma  \cap F$ contained in $B$ is bounded below by a fixed constant independent of $B$. A remarkable feature of this result is that the harmonic measures do not need to satisfy any doubling condition. In the particular case that $\Omega_1$ and $\Omega_2$ are complementary NTA domains, we obtain a geometric characterization of
the $A_\infty$ condition between the respective harmonic harmonic measures of $\Omega_1$ and $\Omega_2$.
\end{abstract}

\maketitle   



 \tableofcontents




\section{Introduction}

In the current manuscript we study a quantitative version of the two-phase problem for harmonic measure proved by the authors in \cite{AMT16}, and the authors and Volberg in \cite{AMTV16}. We assume that we have two disjoint domains in $\R^{n+1}$ (i.e., open and connected sets) whose boundaries satisfy the capacity density condition (CDC) and have non-trivial intersection. If $B$ is a ball centered at a point of their common boundary and the respective harmonic measures with poles in $\frac{1}{2}B$ satisfy an $A_\infty$-type condition only at the top level $B$ (see \eqref{eq:Amu1} for the precise definition), then we show that there exists a subset of the common boundary in $B$  that captures a ``big piece" of $B$ in harmonic measure and is contained in a uniformly rectifiable set.

To state our results in more detail, we need now to introduce some further definitions and notation. Given an open set $\om \subsetneq \R^{n+1}$, $n \geq 1$, we denote $\hm_{\om}^x$ the harmonic measure in $\om$ with pole at $x \in \om$.

We say that $\Omega$ has the {\it capacity density condition (CDC)} if, for all $x\in \d\Omega$ and $0<r<\diam\Omega$, 
\[\textup{Cap}(\overline{{B}(x,r)} \cap \Omega^c, B(x,2r)) \gtrsim r^{n-1},\] 
where Cap$(\cdot,\cdot)$ stands for the variational $2$-capacity of the condenser $(\cdot,\cdot)$ (see \cite[p. 27]{HKM} for the definition).

A set $E\subset \R^d$ is called $n$-{\textit {rectifiable}} if there are Lipschitz maps
$f_i:\R^n\to\R^d$, $i=1,2,\ldots$, such that 
\begin{equation}\label{eq001}
\HH^n\biggl(E\setminus\bigcup_i f_i(\R^n)\biggr) = 0,
\end{equation}
where $\HH^n$ stands for the $n$-dimensional Hausdorff measure. 

A set $E\subset\R^{d}$ is called $n$-AD-{\textit {regular}} (or just AD-regular or Ahlfors-David regular) if there exists some
constant $C_{0}>0$ such that
$$C_0^{-1}r^n\leq \HH^n(B(x,r)\cap E)\leq C_0\,r^n\quad \mbox{ for all $x\in
E$ and $0<r\leq \diam E$.}$$
Note that any domain with $n$-AD-regular boundary is a CDC domain.

The set $E\subset\R^{d}$ is  uniformly  $n$-{\textit {rectifiable}} if it is 
$n$-AD-regular and
there exist constants $\theta, M >0$ such that for all $x \in E$ and all $0<r\leq \diam E$ 
there is a Lipschitz mapping $g$ from the ball $B_n(0,r)$ in $\R^{n}$ to $\R^d$ with $\text{Lip}(g) \leq M$ such that
$$
\HH^n (E\cap B(x,r)\cap g(B_{n}(0,r)))\geq \theta r^{n}.$$


The analogous notions for  measures are the following. 
A Radon measure $\mu$ on $\R^d$ is $n$-{\textit {rectifiable}} if it vanishes outside  an $n$-rectifiable
set $E\subset\R^d$ and if moreover $\mu$ is absolutely continuous with respect to $\HH^n|_E$.
On the other hand, $\mu$ is called $n$-AD-{\textit {regular}} if it is of the form $\mu=g\,\HH^n|_E$, where 
$E$ is $n$-AD-regular and $g:E\to (0,+\infty)$ satisfies $g(x)\sim1 $ for all $x\in E$, with the implicit constant independent of $x$.
If, moreover, $E$ is uniformly $n$-rectifiable, then $\mu$ is called uniformly $n$-{\textit {rectifiable}}.

{ If  $\Omega\subset\R^{n+1}$ is an open set and $c \in (0,\frac{1}{2})$, we say that  a point $x \in \om$  is a $c$-{\textit corkscrew point} for the ball $B(\xi,r)$ with
 $\xi \in\partial\Omega$ and $0<r\leq\diam(\Omega)$, if the ball $B(x,c r)\subset \Omega\cap B(\xi,r)$
with radius $r'\sim r$, with the implicit constant independent of $\xi$ and $r$. }

We are now ready to state our main theorem.

\begin{theorem} \label{teo1}
For $n\geq 1$, let $\om_1, \om_2\subsetneq \R^{n+1}$ be two  open and connected sets satisfying the CDC that are disjoint. Assume also that $F:=\d\om_{1}\cap \d\om_{2} \neq \varnothing$. Suppose that there exist $\ve, \ve' \in (0,1)$ such that if $B$ is a ball centered at $F$ with $\diam B  \leq \min(\diam \d\om_1, \diam \d\om_2)$, there exist points $x_i \in \frac{1}{4}B \cap\om_i$ for $i=1,2$, so that the following holds: for any $E \subset B$,
\begin{equation}\label{eq:Amu1}
\mbox{if}\quad \omega_1^{x_1}(E)\leq \ve\, \omega_1^{x_1}(B), \quad\text{ then }\quad \omega_2^{x_2}(E)\leq \ve'\,\omega_2^{x_2}(2B).
\end{equation} 
If $\ve'$ is small enough (depending only on $n$ and the CDC constant), 
then there exist $\theta_i \in (0,1)$ and a uniformly rectifiable set $\Sigma_B \subset \R^{n+1}$ such that 
$$\omega_i^{x_i}(\Sigma_B \cap F \cap B) \geq \theta_i, \quad i=1,2.$$ 
Moreover, {if $x_{1}$ is a corkscrew point for $\frac{1}{4} B$, then  there exists $c>0$ so that }
\begin{equation}
\label{e:bigpiece}
\cH^{n}(\Sigma_{B}\cap F\cap B)\geq c \,r(B)^{n}. 
\end{equation}
\end{theorem}

This problem can be seen as the ``two-phase" analogue of the ``one-phase" free boundary problems for harmonic measure that have been extensively studied over the last few years in connection with uniform rectifiability. In particular, the usual standing assumptions are that the boundary of the domain is $n$-AD-regular and the domain satisfies the corkscrew condition. Then if harmonic measure satisfies an $A_\infty$-type condition with respect to the ``surface measure" 
(in the same manner as in \eqref{eq:Amu1}) then the boundary of the domain is uniformly rectifiable. This was first proved by Hofmann, Martell and Uriarte-Tuero in \cite{HMU} when the domain is also uniform and was further extended by Hofmann and Martell in \cite{HMIV}\footnote{This preprint is contained in the published paper \cite{HLMN}.} in the context described earlier. In the absence of any additional geometric assumption, the second and third named authors proved in \cite{MT15} that if harmonic measure satisfies a weak-$A_\infty$-type condition with respect to a Frostmann measure supported on the boundary of $\om$, the Riesz transform is bounded from $L^2(\mu)$ to $L^2(\mu)$. This can still be considered as a rather quantitative piece of information, since if $\mu$ is AD-regular this implies that $\mu$ is a uniformly rectifiable measure by the work of the third author with Nazarov and Volberg \cite{NTV14-Acta}. Although, in the absence of lower AD-regularity,  it only shows that $\mu$ is $n$-rectifiable by \cite{NToV14-pubmat}. For elliptic operators, similar results have been proved in uniform AD-regular domains in \cite{HMT} and under the abovementioned standing assumptions in \cite{AGMT}.

On the other hand, in  \cite{KPT09}, Kenig, Preiss and Toro studied the two-phase problem in NTA domains under rather qualitative assumptions, i.e., when $\hm_{1}$ and  $\hm_{2}$ are absolutely continuous and showed that Hausdorff dimension of harmonic measure is exactly $n$. Recently, the three of us in \cite{AMT16} and in collaboration with Volberg in \cite{AMTV16}, we proved that in general domains mutual absolute continuity of the two harmonic measures implies rectifiability. For more details concerning the history of such problems we refer for instance to the introduction of \cite{AMT16}. In the current paper we quantify the assumption of absolute continuity and  obtain  big pieces (in terms of harmonic measure) of uniformly rectifiable sets.

{ We would like to highlight the fact that in our main result we do not assume that $\om$ is a corkscrew domain, while we relax the  measure theoretic assumption that the boundary $\d \om$ is $n$-AD-regular to the potential theoretic one that the domain $\om$ satisfies a 2-sided CDC. This is a big difference from previous results since by a deep theorem of Lewis \cite{Lewis}, it follows that $\Omega$ is a CDC domain
if and only if there exists some $\ve>0$ and some $R>0$ such that
$$\HH_\infty^{n-1+\ve}(B(x,r)\backslash \Omega) \sim r^{n-1+\ve}
\quad \mbox{ for all $x\in\partial\Omega$ and all $0<r\leq R$,}$$
where $\HH_\infty^s$ stands for the $s$-dimensional Hausdorff content. That is, we only require the complement to be large in some sense, while the domain could be very irregular. So, we no longer have the doubling condition of the ``surface" measure which has been crucial to prove uniform rectifiability in the one phase problems. To our knowledge, this is the first time that such a quantitative result is obtained in the absence of Ahlfors-David regularity.

We should say more about the proof of Theorem \ref{teo1} in comparison with Bishop's conjecture that was solved in \cite{AMT16} and \cite{AMTV16} which accounts for a qualitative version of our theorem. Namely, if the harmonic measures $\hm_1$ and $\hm_2$ associated with two disjoint domains $\om_1$ and $\om_2$ such that $\d \om_1 \cap \d \om_2 \neq \varnothing$ are mutually absolutely continuous on a subset of their common boundary then the set is rectifiable. The main difficulty of the problem is that the surface measure of the boundary is not assumed to be locally finite (otherwise the result would follow from \cite{AHM3TV}) and there may be points where the harmonic measure has zero density. The idea there was to use the result from \cite{AHM3TV} for the points of positive density and the rectifiability criterion of  Girela-Sarri{\'o}n and the third named author \cite{GT} for the points of zero density.  A step of critical importance for the use of the latter criterion is the reduction to flat balls that are far away from the poles of the harmonic measures $\hm_i$. This was done by means of blow-up arguments and tangent measures. In particular, it was shown that for $\hm_{i}$-a.e. point $\xi \in \d \om_1 \cap \d \om_2$  for which the Radon-Nikodym derivative $\frac{d\hm_{1}}{d\hm_{2}}$ was positive and finite, all the tangent measures of $\hm_i$ at $\xi$ are flat, which was one of the main novelties in \cite{AMT16}. Then using the monotonicity formula of Alt-Caffarelli-Friedman and non-homogeneous Calder\'on-Zygmund theory we can prove the assumptions of the main theorem in  \cite{GT} and thus, obtain rectifiability.

In the current paper, we follow the same strategy but due to the fact that we ask for a quantitative result, the proof is  different and requires more delicate arguments in order to prove the estimates for the Riesz transform. In a sense, the main enemy is still the same as in \cite{AMT16} and \cite{AMTV16}, i.e., the lack of  Ahlfors-David regularity and the possible abundance of points $x \in \d \om_1 \cap \d \om_2$ for which $\omega_{i}(B(x,r)) \ll r^n$ for $r< r_x$.  In contrast to \cite{AMT16} where the set of points with positive density can be handled via a direct application of \cite{AHM3TV}, we cannot do the same for the set of points for which $\hm_i (B(x,r)) \approx r^n$ by using \cite{MT15} or a Tb theorem and \cite{NTV14-Acta}. Instead, we need to pick up our ``good" set where the Radon-Nikodym derivative $\frac{d\hm_2}{d\hm_1} \approx 1$ to capture a very big piece in measure, i.e., all but $\ve\%$ of the total harmonic measure, and use a theorem that shows $L^2$-boundedness of the Riesz transform on  ``big pieces"  which was ``hidden" (and not stated as such)  in the proof of Theorem 2.1 in \cite{NToV14-pubmat} (see Sections \ref{sec:4} and \ref{sec:5}). For the points of low (uniform) density, arguing with tangent measures  is not helpful any more and we need to come up with compactness arguments (see Lemma \ref{l:beta-reduction}) that will account for the quantitative substitutes of the blow-up analysis we did in \cite{AMT16} and \cite{AMTV16}. We handle this in Sections \ref{subsec:2.4} and \ref{subsec:2.5}, where we show that any sequence of domains satisfying CDC that contains a fixed corkscrew ball and have the weak-$A_\infty$ condition, has a convergence subsequence to a domain $\om_\infty$ that satisfies the CDC and contains the aforementioned ball as a corkscrew ball. Moreover,  the sequence of Green functions $G_{\om_j}$ converges to the Green function $G_{\om_\infty}$ locally uniformly, the harmonic measures $\hm_{\om_j} \rightharpoonup \hm_{\om_\infty}$, and the weak-$A_\infty$ condition for the two harmonic measures is preserved in $\om_\infty$. So, arguing by contradiction and using \cite{AMTV16} for the limiting domain we can reduce case to flat balls that are far away from the poles. This is one of the main ingredients for the proof of the Riesz transform estimates in Theorem \ref{t:GT}, while flatness is also essential and not  {\it a priori} granted .


It is important to note that the set $\Sigma_{B}$ in Theorem \ref{teo1} need not be the whole boundary, and in fact the boundary in $B$ can still have infinite measure. 

\begin{proposition}\label{p:example}
There are disjoint domains $\Omega_{1}$ and $\Omega_{2}\subseteq \mathbb{C}$ so that $\d\Omega_{1}=\d\Omega_{2}$ and \eqref{eq:Amu1} holds for all $E\subseteq \d\Omega_{1}$, yet $\cH^{1}(\d\Omega_{1})=\infty$ and $\d\Omega_{1}$ contains a purely unrectifiable set of positive $\cH^{1}$-measure, that is, a set $K$ so that $\cH^{1}(\Gamma\cap K)=0$ for any curve $\Gamma$ of finite length. 
\end{proposition}

It is not hard to alter the example so that in fact $\dim \d\Omega_{1}>1$. The domains themselves are not too complicated and their boundaries are constructed in a similar way to a von Koch curve. We will prove Proposition \ref{p:example} in section \ref{sec:7}. We thank Svitlana Mayboroda for asking us at the ICMAT in 2018  whether such an example existed.

\vv

In the final part of the paper we consider the particular case of $2$-sided NTA domains (see Section
\ref{sec:7} for the precise definition). As a corollary of Theorem \ref{teo1}, we obtain the following result:
\begin{theorem}\label{teoNTA}
Let $\Omega_1\subset\R^{n+1}$ be an NTA domain and assume that $\Omega_2=\R^{n+1}\setminus \overline
\Omega_1$ is an NTA domain as well. The following conditions are equivalent:
\begin{itemize}
\item[(a)] $\omega_2\in A_\infty(\omega_1)$.
\item[(b)] both $\omega_1$ and $\omega_2$ have very big pieces of uniformly $n$-rectifiable measures.
\item[(c)] either $\omega_1$ or $\omega_2$ has big pieces of uniformly $n$-rectifiable measures.
\item[(d)] {$\Omega_1$ and $\Omega_2$ have joint big pieces of chord-arc subdomains.}
\end{itemize}
\end{theorem}

We remark that, again, we do not assume the surface measure $\HH^n|_{\partial\Omega_1}$ to be locally finite in the theorem.
For the precise notions of $A_\infty(\omega_1)$,  very big pieces of uniformly $n$-rectifiable measures,
 big pieces of uniformly $n$-rectifiable measures, and joint big pieces of chord-arc subdomains, see Section \ref{sec:8}. For the moment, let us say that, for $n$-AD-regular
measures, having very big pieces of uniformly $n$-rectifiable measures or big pieces of uniformly $n$-rectifiable measures is equivalent
to uniform $n$-rectifiability. So Theorem \ref{teoNTA} can be understood as the two-phase counterpart of the David-Jerison theorem \cite{David-Jerison} in the one-phase case, which asserts that for an NTA 
domain $\Omega\subset\R^{n+1}$ with $n$-AD-regular boundary, the condition $\omega\in A_\infty(\HH^n|_{\partial\Omega})$ is equivalent to the uniform $n$-rectifiability of $\HH^n|_{\partial\Omega}$.

\section{Preliminaries}

\subsection{Harmonic measure and Green function} For a (possibly unbounded) domain $\Omega \subset \R^{n+1}$ and $f:\partial \Omega \to [-\infty, \infty]$, we define the {\it upper class} $\mathcal{U}_f$ of $f$ to be the class of functions $u$ such that $u$ is superharmonic, bounded below, and lower semicontinuous for all points on $\d \Omega$. Similarly, we define the {\it lower class} $\mathcal{L}_f$ of $f$ to be the class of functions $u$ such that $u$ is subharmonic, bounded above, and upper semicontinuous for all points on $\d \Omega$. We define 
$$ \overline{H}_f(x) = \inf\{ u \in \mathcal{U}_f\}\quad  \textup{and}  \quad \underline{H}_f(x) = \sup\{ u \in \mathcal{L}_f\}$$
to be the {\it upper} and {\it lower Perron solution} of $f$ in $\Omega$. If $\mathcal{U}_f =\varnothing$(resp. $\mathcal{L}_f = \varnothing$), then we set $ \overline{H}_f = \infty$ (resp. $\underline{H}_f=-\infty$).
 
 A function $f: \d \om \to [-\infty,\infty]$ is called {\it resolutive} if the upper and the lower Perron solutions of $f$ coincide, and if we denote ${H}_f=\overline{H}_f=\underline{H}_f$ then it is harmonic in $\om$. Notice that if the complement of $\om$ has positive capacity then every continuous function is resolutive.

Let $\EE$ denote the {\it fundamental solution} for the Laplace equation in $\R^{n+1}$, which is defined as
  \begin{equation}
   \mathcal{E}(x)=
    \begin{cases}
      c_n'\,|x|^{1-n}, & \text{if}\,\, n\geq 2 \\
      \frac{1}{2\pi}\,\log\frac{1}{|x|}, & \text{if}\,\,n=1,
    \end{cases}
  \end{equation}
where $c_n'>0$ a dimensional constant.

A {\it Green function} $G_{\Omega}:\Omega\times \Omega\rightarrow[0,\infty]$ for an open set $\Omega\subset \R^{n+1}$ is a function with the following properties: for each $x\in \Omega$, $G_{\Omega}(x,y)=\EE(x-y) + h_{x}(y)$ where $h_{x}$ is harmonic on $\Omega$, and whenever $v_{x}$ is a nonnegative superharmonic function that is the sum of $\EE(x-\cdot)$ and another superharmonic function, then  $v_{x}\geq G_{\Omega}(x,\cdot)$ (\cite[Definition 4.2.3]{Helms}). 

An open subset of $\R^{n+1}$ having a Green function will be called a {\it Greenian} set. The class of domains considered in Theorem \ref{teo1} are always Greenian. Indeed, all open subsets of $\R^{n+1}$ are Greenian for $n\geq 2$ (\cite[Theorem 4.2.10]{Helms}); in the plane, domains with nonpolar boundaries are Greenian by Myrberg's Theorem (see Theorem 5.3.8 on p.\ 133 of \cite{AG}). In particular, if a domain satisfies the CDC, then it is Greenian. 

If $\om$ is a Greenian domain and  $y \in \om$, then if we denote $u_y(x)= {H}_{\EE(\cdot-y)}(x)$ for $x \in \om$ (since $\EE(\cdot-y)$ is resolutive),  by \cite[Definition 4.3.14]{Helms}, we have that
\begin{equation}\label{eq:Green-repr}
G_{\Omega}(x,y)=\EE(x-y)-u_y(x), \quad \textup{for}\,\, x \in \om.
\end{equation}
Note that in \cite{Helms} it is only stated for bounded domains $\om$ but it is enough for the  complement  of $\om$ to have positive capacity.

For each $x \in \Omega$, one can construct the {\it harmonic measure} $\omega^x_\Omega$ (see e.g. \cite[Section 11]{HKM},  \cite[p. 172]{AG}, or \cite[p. 217]{Helms}). In fact, if $\om$ is a Greenian domain, for any continuous function $f$, since the mapping $f \mapsto {H}_f$ is a positive linear functional, then by Riesz representation theorem, it is given by 
$${H}_f(x) = \int_{\partial_\infty \Omega} f(y) \,d\omega_\Omega^x(y),$$
where $\d_\infty\Omega=\d\Omega$ if $\Omega$ is bounded and $\d_\infty\Omega=\d\Omega\cup\{\infty\}$ otherwise.
 Remark  that constant functions are continuous and  since $H_1(x)=1$, for any $x \in \Omega$, we have that $\omega_\Omega^x(\partial_\infty \Omega)=1$, for any $x \in \Omega$.

Given a signed Radon measure $\nu$ in $\R^{n+1}$ we consider the $n$-dimensional Riesz
transform
$$\RR\nu(x) = \int \frac{x-y}{|x-y|^{n+1}}\,d\nu(y),$$
whenever the integral makes sense. For $\ve>0$, its $\ve$-truncated version is given by 
$$\RR_\ve \nu(x) = \int_{|x-y|>\ve} \frac{x-y}{|x-y|^{n+1}}\,d\nu(y).$$
and we set
$$\RR_{*,\delta} \nu(x)= \sup_{\ve>\delta} |\RR_\ve \nu(x)|.$$
Notice that the kernel of the Riesz transform is
\begin{equation}\label{eqker}
K(x) = c_n\,\nabla \EE(x),
\end{equation}
for a suitable absolute constant $c_n$. 

\vv

The following result, sometimes known as ``Bourgain's estimate", also holds. For a proof see e.g. Lemma 11.21 in \cite{HKM}.

\begin{lemma}\label{lem:bourgain}
Let $\Omega\subsetneq \bR^{n+1}$, $n\geq 1$,  be an open set that satisfies the CDC,  $x\in \d\Omega$, and $0<r\leq\diam \partial\Omega$.  Then 
\begin{equation}\label{eq:bourgain}
\hm^y(B(x,r))\geq c >0, \;\; \mbox{ for all }y\in \Omega\cap B(x,\tfrac{r}{2}),
\end{equation}
where $c$ depends on $n$ and the CDC constant.
\end{lemma}
\vv


\vv

The next result is also standard and follows from Lemma \ref{lem:bourgain}. For a proof see e.g. Lemma 2.3 in \cite{AM15}.

\begin{lemma} \label{lem:Holder-hm}
Let $\Omega\subsetneq\R^{n+1}$, $n\geq 1$,  be an open set that satisfies the CDC and let $x\in \d\Omega$. Then there is $\alpha>0$ so that for all $0<r<\diam \Omega$,
\begin{equation}\label{eq:hm-holder}
 \omega^{y}({B}(x,r)^{c})\lesssim \ps{\frac{|x-y|}{r}}^{\alpha},\quad \mbox{ for all } y\in \Omega\cap B(x,r),
 \end{equation}
where $\alpha$ and the implicit constant depend on $n$ and the CDC constant.
\end{lemma}

As a corollary   of Lemma \ref{lem:Holder-hm}, maximum principle, and standard properties of harmonic functions, one can obtain the following lemma.
\begin{lemma}\label{lem:Holder-u-van}
Let $\Omega\subsetneq\R^{n+1}$ be an open set that satisfies the CDC. Let $x\in\partial\Omega$ and $0<r<\diam \Omega$.
Let $u$ be a non-negative harmonic function in $B(x,4r)\cap \Omega$ and continuous in $B(x,4r)\cap \overline\Omega$
so that $u\equiv 0$ in $\partial\Omega\cap B(x,4r)$. Then extending $u$ by $0$ in $B(x,4r)\setminus \overline\Omega$,
there exists a constant $\alpha>0$ such that
$$u(y)\leq C\,\left(\frac{\delta_\Omega(y)}r\right)^\alpha \!\sup_{B(x,2r)}u
\leq C\,\left(\frac{\delta_\Omega(y)}r\right)^\alpha \;\avint_{B(x,4r)}u,
\quad \mbox{for all $y\in B(x,r)$,}$$
where $C$ and $\alpha$ depend on $n$ and the CDC constant.
In particular, $u$ is $\alpha$-H\"older continuous in $B(x,r)$.
\end{lemma}

\begin{rem}
\label{r:greeneverywhere}
Note that by Lemma \ref{lem:Holder-u-van}, the Green function $G_{\Omega}(x_{0},\cdot)$ can be extended continuously to all of $\R^{n+1}\backslash \{0\}$ by setting $G_{\Omega}(x_{0},x)=0$ for $x\in \Omega^{c}$. We will assume this throughout the paper. 
\end{rem}

The following lemma concerns the rate of decay at infinity of a bounded harmonic function vanishing outside a ball centered at the boundary. 

\begin{lemma}\cite[Lemma 2.11]{AGMT16}\label{lem:u-decay-infinty}
Let $\Omega\subsetneq \R^{n+1}$, $n\geq 1$, be an open set that satisfies the CDC.
Let $u$ be a bounded, harmonic function in $\Omega$, vanishing at $\infty$, and let $B$ be a ball centered at
$\partial\Omega$. Suppose that $u$ vanishes continuously in $\partial\Omega\setminus B$. Then, there is a constant
$\alpha>0$ such that
\begin{equation}\label{eqdjl125}
|u(x)| \lesssim \frac{r(B)^{n-1+\alpha}}{\bigl(r(B) + \dist(x,B)\bigr)^{n-1+\alpha}}\,\|u\|_{L^\infty(\Omega \cap (3B \setminus 2B))}.
\end{equation}
Both $\alpha$ and the constant implicit in the above estimate depend only on $n$ and the CDC constant.
\end{lemma}

{
\begin{proof}
This lemma was originally stated for domains with $n$-AD-regular boundary and $n\geq 2$. Nevertheless, the same proof  goes through unchanged for $n=1$ and under the assumption that $\om$ satisfies the CDC. Indeed,  we only used AD-regularity  to apply Lemma \ref{lem:bourgain} and deduce that the boundary is Wiener regular and there exists a constant $C>1$ such that $B(x_B, Cr) \setminus B(x_B,r) \neq \varnothing$ for every $r>0$, where $x_B$ is the center of $B$. All those ingredients are still at our disposal if we only assume the CDC.
\end{proof}
}

\vv

\subsection{P-doubling}

 Given $\gamma>0$, a Borel measure $\mu$ and a ball $B\subset\R^{n+1}$, we denote
$$P_{\gamma,\mu}(B) = \sum_{j\geq0} 2^{-j\gamma}\,\Theta^n_\mu(2^jB),$$
where $\Theta^n_\mu(B) = \frac{\mu(B)}{r(B)^n}$. We also write $P_{1,\mu}(B) = P_{\mu}(B)$. Note that $P_{\gamma_1,\mu}(B) \leq P_{\gamma_2,\mu}(B)$ if $\gamma_1>\gamma_2$, so in particular, \begin{equation}
\label{e:pmonotone}
P_{\mu}(B)\leq P_{\gamma,\mu}(B) \mbox{ for $\gamma<1$}.
\end{equation} 
It is immediate to check
that if $\|\mu\|<\infty$, then $P_{\gamma,\mu}(B)<\infty$ for any ball $B$. Indeed, we just take into
account that
\begin{equation}
P_{\gamma,\mu}(B)
=\sum_{j\geq0} 2^{-j\gamma}\,\Theta^n_\mu(2^jB)\leq 
\sum_{j\geq0} 2^{-j\gamma}\,\frac{\|\mu\|}{(2^j\,r(B))^n}<\infty.
\end{equation}

Given $a,\gamma>0$, we say that a ball $B$ is {\it $a$-$P_{\gamma,\mu}$-doubling} if
$$P_{\gamma,\mu}(B) \leq a \,\Theta^n_\mu(B).$$ Note that by \eqref{e:pmonotone}, if $\gamma<1$ and $B$ is $a$-$P_{\gamma,\mu}$-doubling, then it is also $a$-$P_{\mu}$-doubling.

\begin{lemma}\cite[Lemma 6.1]{AMT16}
\label{l:adoubling}
There is $\gamma_{0}\in (0,1)$ so that the following holds. Let $\Omega\subset \bR^{n+1}$ be any domain and $\omega$ its harmonic measure. For all $\gamma>\gamma_{0}$, there exists some big enough constant $a=a(\gamma,n)>0$ such that  for $\omega$-a.e.
$x\in\R^{n+1}$ there exists a sequence of $a$-$P_{\gamma,\omega}$-doubling balls $B(x,r_i)$, with
$r_i\to0$ as $i\to\infty$.
\end{lemma}

\subsection{Monotonicity}

The following theorem contains the Alt-Caffarelli-Friedman monotonicity formula:

\begin{theorem} \label{t:ACF} \cite[Theorem 12.3]{CS} Let $B(x,R)\subset \R^{n+1}$, and let $u_1,u_2\in
W^{1,2}(B(x,R))\cap C(B(x,R))$ be nonnegative subharmonic functions. Suppose that $u_1(x)=u_2(x)=0$ and
that $u_1\cdot u_2\equiv 0$. Set
\begin{equation}\label{eq*123}
J(x,r) = \left(\frac{1}{r^{2}} \int_{B(x,r)} \frac{|\grad u_1(y)|^{2}}{|y-x|^{n-1}}dy\right)\cdot \left(\frac{1}{r^{2}} \int_{B(x,r)} \frac{|\grad u_2(y)|^{2}}{|y-x|^{n-1}}dy\right).
\end{equation}
Then $J(x,r)$ is a non-decreasing function of $r\in (0,R)$ and $J(x,r)<\infty$ for all $r\in (0,R)$. That is,
\begin{equation}\label{e:gamma}
J(x,r_{1})\leq J(x,r_{2})<\infty \;\;  \mbox{ for } \;\; 0<r_{1}\leq r_{2}<R.
\end{equation}
\end{theorem}
We intend to apply the preceding theorem to the case when $\om_1$ and
$\om_2$ are disjoint CDC domains, $x\in\partial \om_1\cap \partial \om_2$, with $u_1,u_2$ equal to the Green functions of $\Omega_1,\Omega_2$ with poles at $x_1,x_2$ being corkscrew points for the ball $B(x,r)$ in $\om_1$ and $\om_2$ respectively, extended by $0$ to $(\om_1)^c, (\om_2)^c$.  In this case,  $u_i\in W^{1,2}_{loc}(\R^{n+1}\setminus\{x_i\})\cap C(\R^{n+1}\setminus\{x_i\})$ for $i=1,2$ and so the assumptions of  the preceding theorem are satisfied in any ball which does contain $x_1$ and $x_2$.

The next lemma is shown in \cite[Section 4]{AMT16} for $n \geq 2$. The case $n=1$ can be proved using arguments similar to the ones appeared in \cite[pp.17-19]{AHM3TV}  or \cite[pp.16-17]{MT15}.
\begin{lemma}
\label{l:otherside}
There is $C_{1}>0$ so that the following holds. Let $\om_1, \om_2 \subset \R^{n+1}$ be two disjoint CDC-domains such that $F= \d \om_1 \cap \d \om_2 \neq \varnothing$. Let $\hm_i$ stand for the harmonic measure in $\om_i$ with pole $x_i \in \om_i$ for $i=1,2$, and assume that $\min\{\hm_1(F),\hm_2(F)\}>0$. If $$0<R<\min\{\dist(x_{1},\d\om_{1}),\dist(x_{2},\d\om_{2})\},$$ then for $\xi\in F$ and $r< R/8$,
\begin{equation}\label{e:beurling}
\frac{\omega_{1}(B(\xi,r))}{r^{n}}\,\frac{\omega_{2}(B(\xi,r))}{r^{n}}\lec J(\xi,2r)^{\frac{1}{2}},
\end{equation}
and
\begin{equation}\label{e:otherside}
J(\xi,r)^{\frac{1}{2}}\lec \frac{\omega_{1}(B(\xi,C_{1}r))}{r^{n}}
\frac{\omega_{2}(B(\xi,C_{1}r))}{r^{n}},
\end{equation}
where $J(\xi,r)$ is defined by \rf{eq*123}.
\end{lemma}

\subsection{Limits of harmonic measures}\label{subsec:2.4}

The following lemma will allow us to use compactness arguments with harmonic measures in CDC domains. 

\begin{lemma}\label{limlem}
Let $\Omega_{j}\subset \bR^{n+1}$ be a sequence of CDC domains with the same CDC constants such that $0\in \d\Omega_{j}$, $\inf \diam \d\Omega_{j}>0$, and suppose there is a ball $B_0=B(x_{0}, r_0) \subset \om_j$ for all $j \geq1$. Then there is a connected open set $\Omega_{\infty}^{x_{0}}$ containing $B_0$ so that, after passing to a subsequence,
\begin{enumerate}
\item $G_{\Omega_{j}}(x_{0},\cdot)$ converges uniformly on compact subsets of $\{x_{0}\}^{c}$ to
 $G_{\Omega_{\infty}^{x_{0}}}(x_{0},\cdot)$,
\item   $\omega_{\Omega_{j}}^{x_{0}}\warrow \omega_{\Omega_{\infty}^{x_{0}}}^{x_{0}}$, 
\item and $\Omega_{\infty}^{x_{0}}$ is also a CDC domain with the same constants. 
\end{enumerate}
\end{lemma}

\begin{proof}
Recall that 
{\begin{equation*}
u_j(x) = G_{\om_j}(x, x_0) \lec \begin{cases}
        |x-x_{0}|^{1-n}, & \text{if}\,\, n\geq 2 \\
       \left| \log|x-x_0|\right|, & \text{if}\,\,n=1.
    \end{cases}
\end{equation*}}
Thus, if we extend $u_j$ by zero in $\R^{n+1} \setminus \om$, for every ball $B \subset \R^{n+1}\backslash B_0$  we have that
{\begin{equation}\label{eq:uj-limit}
\sup_{x \in B} \,u_j(x) \lec \begin{cases}
        r_0^{1-n}, & \text{if}\,\, n\geq 2 \\
       \left|  \log r_0 \right|, & \text{if}\,\,n=1
    \end{cases} \quad \textup{for all}\,\, j \geq 1.
\end{equation}
}
 By Lemma \ref{lem:Holder-u-van}, the harmonicity of $G_{\Omega_{j}}(x_{0},\cdot)$ in $\Omega_{j}\backslash \{x_{0}\}$, and \eqref{eq:uj-limit}, one can show that the $u_{j}:=G_{\Omega_{j}}(x_{0},\cdot)$  form an equicontinuous family on each compact subset of $\bR^{n+1}\backslash \{x_{0}\}$. Indeed, if $K$ is a compact subset of $\R^{n+1}$ we may cover it with a finite number of balls $B$ so that $x_{0}\not\in 20B$  and $r(B)<\frac{1}{4}\inf \diam \d\Omega_{j}$. Thus, it is enough to restrict ourselves to such balls. If $x,y\in B$, then there are a few cases:

\begin{enumerate}
\item[Case 1:] If $x,y\not\in \Omega_{j}$, then $|u_{j}(x)-u_{j}(y)|=0$. 
\item[Case 2:]  If $x\in \Omega_{j}$ and $y\not\in \Omega_{j}$, if $z\in [x,y]\cap \d\Omega_{j}$, then by Lemma \ref{lem:Holder-u-van} and \eqref{eq:uj-limit}, and since $x_{0}\not\in B(z,3r(B))$, we have that 
\[
|u_{j}(x)-u_{j}(y)|
=u_{j}(x)
\lec  \ps{\sup_{B(z,2r(B))} u_{j}}|x-z|^{\alpha} \stackrel{\eqref{eq:uj-limit}}{\lec_{r_0}} |x-y|^{\alpha}. \]
\item[Case 3:] Let $x, y \in \Omega_{j}$ and suppose there is $z\in \Omega_{j}^{c}$ with $\dist(z,[x,y])<\frac{1}{3} |x-y|$. Without loss of generality we may  assume that $ |u_{j}(x)|\leq |u_{j}(y)|$. Thus, if $z'\in [x,y]$ is the point that realizes $\dist(z,[x,y])$, then 
\begin{equation}
\label{e:y-z}
|y-z|\leq |y-z'|+|z'-z|
\leq |x-y|+\frac{1}{3}|x-y|=\frac{4}{3}|x-y|\leq \frac{8}{3}r(B).
\end{equation}
Using the latter inequality and that $x_{0}\not\in 20B$, if we denote $c_B$ the center of $B$, we have 
\begin{align*}
|x_0-z| &\geq  |x_0-c_B| -|z-y| -|y-c_B| \\
& \geq 20 r(B) - \frac{8}{3}r(B) - r(B) \\
&> 16 r(B).
\end{align*}
In particular, if $B'=B(z,3r(B))$, then $x_{0}\not\in 5 B'$. Thus, by  Lemma \ref{lem:Holder-u-van} and \eqref{eq:uj-limit},
\begin{align*}
|u_{j}(x)-u_{j}(y)|
& 
\leq 2|u_{j}(y)|  \lec |y-z|^{\alpha} \sup_{2B'}u_{j}\stackrel{\eqref{eq:uj-limit}}{\lec_{r_0}} |y-z|^{\alpha}  
 \stackrel{\eqref{e:y-z}}{\lec_{r_0}}   |x-y|^{\alpha}.
\end{align*}
\item[Case 4:] Let $x, y \in \Omega_{j}$ and suppose that $\dist(z,[x,y])\geq\frac{1}{2} |x-y|$ for every $z\in \Omega_{j}^{c}$. We just use the interior Lipschitz continuity of harmonic functions (since they are  real analytic) and a Harnack chain argument.
\end{enumerate}

Since this holds for all such balls $B$, the functions $u_{j}$ form an equicontinuous 
family of functions on compact subsets of $\bR^{n+1}\backslash \{x_{0}\}$.
This family is also pointwise bounded, by \rf{eq:uj-limit}.
So we may pass to a subsequence so that the $u_j$'s converge uniformly on compact subsets to a function $u_\infty$.  

Let 
\[
\Omega_{\infty}^{x_{0}}=\{x:u_\infty(x)>0\}.
\]

Also recall that by \eqref{eq:Green-repr}, for $x \in \overline\om_j$ it holds that $u_{j}(x)=\EE(x-x_{0}) - h_{j}(x)$, where $h_{j}$ is the harmonic extension of $\EE(\cdot-x_{0})$ in $\Omega_{j}$.  We may extend $h_j$ to be equal to $\EE(\cdot-x_{0})$ in $\R^{n+1} \backslash \overline \om_j$, and hence, the same identity holds in the whole $\R^{n+1}$. Notice that $h_j$ is continuous in $\R^{n+1}$. Therefore, for the same sequence that $u_j \to u_\infty$ locally uniformly in $\bR^{n+1}\backslash \{x_{0}\}$ we have that $h_{j}$ converges uniformly on compact subsets of $\R^{n+1}\backslash \{x_{0}\}$ to some continuous function $h_{\infty}$, that is, 
\[
u_{\infty}(x)= \EE(x-x_{0}) - h_{\infty}(x),\quad \textup{for all}\,\, x \in  \R^{n+1} \backslash\{x_0\}.\]

Moreover, since  $h_j$ is harmonic and bounded in $\om_j$, we may pass to a subsequence so that $h_j \to h_\infty$ uniformly on compact subsets of  $\om_\infty^{x_{0}}$, where $h_\infty$ is  harmonic in $\om_\infty^{x_{0}}$. 
By definition, $u_{\infty}\equiv 0$ on $\d\Omega_{\infty}^{x_{0}}$, and so $h_{\infty}(x)= \EE(x-x_{0})$ on $\d\Omega_{\infty}^{x_{0}}$. Thus, since $h_\infty$ is harmonic and continuous up to the boundary of $\Omega_{\infty}^{x_{0}}$, and $\EE(\cdot-x_{0})$ is bounded on $\d \om^{x_0}_\infty$, we infer that $h_\infty = H_{\EE(\cdot-x_{0})}$ (see e.g. \cite[Theorem 3.6.10]{Helms}), and so  by \eqref{eq:Green-repr},  $u_{\infty}=G_{\Omega_{\infty}^{x_{0}}}(x_{0},\cdot)$ vanishing continuously on the boundary.
In particular, this implies that $\Omega_{\infty}^{x_{0}}$ is Wiener regular and connected.

Since $\omega_{\Omega_{j}}^{x_{0}}$ is a probability measure, we can pass to a subsequence so that $\omega_{\Omega_{j}}^{x_{0}}$ converges weakly to a Radon measure $\omega_{\infty}^{x_{0}}$. Let $\phi\in C_{c}^{\infty}(\bR^{n+1})$ be such that $\phi(x_{0})=0$. Then
\begin{equation}\label{phiom}
\int \phi \,d\omega_{\infty}^{x_{0}} 
= \lim_{j\rightarrow\infty} \int \phi \,d\omega_{\Omega_{j}}^{x_{0}} 
=\lim_{j\rightarrow\infty} \int G_{\Omega_{j}}(y,x_0) \Delta\phi \,dy
=\int G_{\Omega_{\infty}^{x_{0}}}(y,x_0) \Delta\phi \,dy.
\end{equation}
Hence, $\omega_{\infty}^{x_{0}}$ is the harmonic measure for the domain $\Omega_{\infty}^{x_{0}}$ with pole at $x_{0}$. 

Finally, let $B(\xi,r)$ be a ball centered on $\xi \in \d\Omega_{\infty}^{x_{0}}$ so that $r<\diam \d\Omega_{\infty}^{x_{0}}$. 

 Since $\supp \omega_{\Omega_{\infty}^{x_{0}}}=\d\Omega_{\infty}^{x_{0}}$, $\supp \omega_{\Omega_{j}^{x_{0}}}=\d\Omega_{j}$ (see Lemma 4.7 in \cite{AAM} for the detailed proof), and $\omega_{\Omega_{j}}^{x_{0}}\warrow \omega_{\Omega_{\infty}^{x_{0}}}^{x_{0}}$, we know there exists $\xi_j \in \d \om_j$ such that $\xi_j \to \xi$. 
 
If we pick $\zeta_{1},\zeta_{2}\in \d\Omega_{\infty}^{x_{0}}$ with $r<|\zeta_{1}-\zeta_{2}|$, then this pair is the limit of points $\zeta_{1,j},\zeta_{2,j}\in \d\Omega_{j}$, so in fact
\[
r<|\zeta_{1}-\zeta_{2}|\leq \liminf_{j} \diam \d\Omega_{j}.
\]
Thus, we will assume below that $j$ is large enough that  $r<\diam \d\Omega_{j}$.
 
 Let 
\begin{equation}
\label{e:BjBxir}
B_{j}= B(\xi_{j},r-|\xi-\xi_{j}|)\subset B(\xi,r).
\end{equation}
Since $\Omega_{j}$ has the CDC, by the results of \cite{Lewis} (as mentioned in the introduction), there are $c>0$ and $s\in (n-1,n+1)$ so that
\begin{equation}
\label{e:Bjbig}
\cH^{s}_{\infty}(\cnj{B}_{j}\backslash \Omega_{j})\geq c (r-|\xi-\xi_{j}|)^{s}.
\end{equation}
Also, if $y_{j}\in \Omega_{j}^{c}$ converges to $y\in \mathbb{R}^{n+1}$, then 
\[
G_{\Omega_{\infty}^{x_{0}}}(x_{0},y)=\lim_{j\rightarrow\infty} G_{\Omega_{j}}(x_{0},y_{j})=0,\]
and so $y\in (\Omega_{\infty}^{x_{0}})^{c}$. In particular, by a compactness argument, for every $\ve>0$, there is $j_0$ so that for $j\geq j_0$ it holds
\[
\cnj{B(\xi,r)}\backslash \Omega_{j}\subset (\cnj{B(\xi,r)}\backslash \Omega_{\infty}^{x_{0}})_{\ve},\]
where $(K)_\ve$ stands for the $\ve$-neighborhood of a set $K$.
Using this and the compactness of $\cnj{B(\xi,r)}\backslash \Omega_{\infty}^{x_{0}}$, it is not hard to show that 
\begin{align*}
\cH^{s}_{\infty}(\cnj{B(\xi,r)}\backslash \Omega_{\infty}^{x_{0}}) 
& \geq \limsup_{j\rightarrow \infty}\cH^{s}_{\infty}(\cnj{B(\xi,r)}\backslash \Omega_{j}) \\
& \stackrel{\eqref{e:BjBxir}}{\geq} \limsup_{j\rightarrow \infty}\cH^{s}_{\infty}(\cnj{B_{j}}\backslash \Omega_{j}) \\
& \stackrel{\eqref{e:Bjbig}}{\geq} \lim_{j\rightarrow\infty} c(r-|\xi-\xi_{j}|)^{s}=cr^{s}.
\end{align*}
Since this holds for all such $x\in \d\Omega_{\infty}^{x_{0}}$ and $0<r<\diam \d\Omega_{\infty}^{x_{0}}$,  $\Omega_{\infty}^{x_{0}}$ is also a  CDC domain.

\end{proof}
One thing we would like to do below is compute a limit of the form $\omega_{j}(B)$ where $\omega_{j}$ is a convergent sequence of harmonic measures and $B$ is a ball. This is only possible if we know the boundary of the ball is null with respect to the limiting measure. However, since we know the limiting measure is also harmonic measure, the next lemma will allow us to do this for essentially any ball.

\begin{lemma}
\label{l:nullboundary}
Let $\omega$ be harmonic measure for some domain $\om$ in $\R^{n+1}$, $n \geq 1$. Then there are at most countably many balls $B_{1},B_{2},...$ for which $\omega(\d B_{j})>0$. 
\end{lemma}

\begin{proof}
Suppose there was an uncountable set  $A$ of balls $B$ for which $\omega( \d B)>0$. Then there is $t>0$ so that if 
\[
A':=\{B\in A: \omega( B)>t\},
\]
then $A'$ is uncountable. Let $B_{1},B_{2},...$ be any infinite sequence of distinct balls in $A'$. Note that if $ \d B_{i}\cap \d B_{j} \neq \varnothing$ and $i\neq j$, then $\dim \d B_{i}\cap \d B_{j}\leq n-1$. This implies $ \d B_{i}\cap \d B_{j}$ has $2$-capacity zero \cite[Theorem 2.27]{HKM}, hence is a polar set for $\omega$ \cite[Theorem 10.1]{HKM} and polar sets have harmonic measure zero \cite[Theorem 11.15]{HKM}. Thus, if we set
\[
D_{i} = \d B_{i}\backslash \bigcup_{j\neq i } \d B_{j},\]
for an uncountable number of indices $i \in I$,
then  $\{D_{i}\}_{i \in I}$ is a pairwise disjoint family of sets for which $\omega(D_{i})>t$ and its union equals to $\bigcup_i \d B_{i}$. Hence, 
\[
1
\geq \omega(\R^{n+1})
\geq \omega\ps{\bigcup \d B_{i}}
=\sum_{i} \omega(D_{i})=\infty,\]
which is a contradiction. 
\end{proof}

\subsection{The weak-$\wt A_{\infty}(B)$ property for measures}\label{subsec:2.5}

Since the harmonic measures we will be dealing with have the weak-$\wt A_{\infty}(B)$ property, the following lemmas will be useful. Though we will use these results specifically for harmonic measure, we will state them more generally for Radon measures.

\begin{definition}
Let $\mu$ and $\nu$ be Radon measures. We say $\nu\in \textup{weak-}\wt A_{\infty}(\mu,B)$ for some ball $B$ if there exist $\ve, \ve'\in (0,1)$ so that, for all $E\subset B$,
\[
\mu(E)<\ve \mu(B) 
\;\; \Rightarrow \;\;
\nu(E)<\ve'\nu(2B).
\]
\end{definition}

\vv

\begin{lemma}\label{l:goodset}
Let $\mu$ and $\nu$ be { finite} Borel measures in $\R^{n+1}$. Let $B\subset \R^{n+1}$ be some ball such that
$\mu(B)\sim\mu(2B)$ and $\nu(B)\sim\nu(2B)$, and assume that $\nu\in \textup{weak-}\wt A_{\infty}(\mu,B)$, with comstants $\ve,\ve'\in(0,1)$.
Then for $A>1$ large enough (depending on $\ve$ and $\ve'$), there is a set $G\subset B$ so that 
$\nu(B\backslash G)<2\ve' \nu(B)$, and 
\begin{equation}\label{eq:mu-nu-comparable}
A^{-1} \frac{\mu(B)}{\nu(B)} \leq   \frac{\mu(B(x,r))}{\nu(B(x,r))}\leq A\frac{\mu(B)}{\nu(B)} \quad  \mbox{ for all }x\in G, 0<r<r(B).
\end{equation}
\end{lemma}

\begin{proof}
Without loss of generality, we may assume $\mu(B)=\nu(B)=1$. 

We  consider  the truncated centered Hardy-Littlewood maximal operator 
$$
\wt{\mathcal{M}}_\mu\nu(x)= \sup_{0<r\leq r(B)} \frac{\nu(B(x,r))}{\mu(B(x,r))}.
$$
Recall that this is bounded from { the space of finite Borel measures} $ M(\R^{n+1})$ to  $L^{1,\infty}(\mu)$. Note also that 
\[
F_{1} := \{x\in B: \wt{\mathcal{M}}_{\mu}\nu(x)>A\} = \{x\in B: \wt{\mathcal{M}}_{\mu}(\chi_{2B}\nu)(x)>A\},
\]
and thus $\mu(F_{1})\lec A^{-1}\,\nu(2B)\sim A^{-1}$. We can pick $A$ large enough so that $\mu(F_{1})<\ve$. By the $\textup{weak-}\wt A_{\infty}$ property,
 this implies $\nu(F_{1})<{\ve'}$. 

Now let 
\[
F_{2} = \{x\in B: \wt{\mathcal{M}}_{\nu}\mu(x)>A\}
.\]
Analogously, we get $\nu(F_{2})\lec A^{-1}$,  and by picking $A$ large enough, 
we can guarantee that 
$$\nu(F_1\cup F_2) \leq \ve' + C\,A^{-1} < 2\ve'.$$
 If we set 
$$G=B\backslash (F_{1}\cup F_{2}),$$ 
then we clearly have $\nu(B\backslash G)<2\ve'$ and $G$ satisfies the conclusions of the lemma. 
\end{proof}

\vv

\begin{rem}\label{remdob}
Under the assumptions of Theorem \ref{teo1},  from Lemma \ref{lem:bourgain} and the fact that the corkscrew
points $x_i$ belong to $\frac14B$, we infer that
$$\omega_i^{x_i}(\tfrac12 B) \sim \omega_i^{x_i}(B) \sim \omega_i^{x_i}(2B) \sim 1.$$ 
{ Moreover, since in the statement of Theorem \ref{teo1} we ask for $\ve'$ to be small enough, we can always switch between $\omega_i^{x_i}(B)$ and $\omega_i^{x_i}(2B)$.}

In particular, the ball $B$ is doubling for both measures $\omega_1^{x_1}$, $\omega_2^{x_2}$, and the preceding lemma can be applied with $\mu=\omega_1^{x_1}$ and $\nu=\omega_2^{x_2}$. In this case, it turns out that $G\subset F=\d\om_1\cap\d\om_2$ (with the possible exception of a set of zero measure $\omega_i^{x_i}$), since $\omega_1^{x_1}$ and $\omega_2^{x_2}$ are mutually singular
out of $F$. 
\end{rem}
\vv

\begin{lemma}\label{l:afinlim}
Assume that $\mu_{j}\warrow \mu$, $\nu_{j}\warrow \nu$, and that $B$ is a ball such that $\mu(\d B)=\nu(\d B)=0$ and for all $E\subset B$ compact,
\[
\mu_{j}(E)<\ve \mu_{j}(B)  \;\; \Rightarrow \nu_{j}(E)<\ve' \nu_{j}(B).
\]
Then for all $E\subset B$ compact,
\[
\mu(E)<\ve \mu(B)  \;\; \Rightarrow \nu(E)\leq \ve' \nu(B).
\]
\end{lemma}

\begin{proof}
Let $E$ be compact such that 
\[
\mu(E)<\ve \mu(B).
\]
Notice that for all points $x\in E$, there are at most countably many radii $r>0$ for which $\mu(\d B(x,r))+\nu(\d B(x,r))>0$. Thus, by  Vitali covering theorem and compactness, we may find finitely many open balls $B_{1},...,B_{N}\subset B$ (recall $B$ is open and $E$ is compact) so that
\[
E\subset F:= \bigcup_{i=1}^{N} B_{i}, \;\;\; \mu(\d B_{i})=\nu(\d B_{i})=0 \mbox{ for }i=1,...,N,
\]
and
\[
\mu \ps{F\backslash E}< \ve \mu(B) - \mu(E).
\]
Thus,
\[
\ve \mu(B)> \mu(F)=\mu(\cnj{F}).
\]
In particular, $\mu(\d F)=\nu(\d F)=0$.

Recall from \cite[Chapter 1, Exercise 9]{Mattila} that if $A$ is a set so that $\mu(\d A)=0$ and $\mu_{j}\warrow \mu$, then $\mu_{j}(A)\rightarrow \mu(A)$.  Hence,
\[
0<\ve \mu(B)- \mu(F)
=\lim_{j} (\ve \mu_{j}(B)-\mu_{j}(F)),\]
and so, for sufficiently large $j$, using  the $ \textup{weak-}\wt A_{\infty}(\mu_j,B)$ property of $\nu_j$, we obtain 
\[
\nu_{j}(F)<\ve ' \nu_{j}(B).
\]
Thus,
\[
\nu(E)\leq \nu(F)
=\lim_{j} \nu_{j}(F)
\leq \ve' \lim_{j} \nu_{j}(B)
=\ve' \nu(B).
\]

\end{proof}

\section{Initial reductions}\label{sec:initial}

\subsection{Reduction to poles that are corkscrew points}

We will show that the CDC of the domains in Theorem \ref{teo1}  ensures that, if the poles $x_1$ and $x_2$ for  of the harmonic measures $\omega_1$ and $\omega_2$  aren't known to be corkscrew points for $\frac{1}{4}B$, they are  in fact corkscrew points for some other appropriate ball.

\begin{lemma}
Under the assumptions of Theorem \ref{teo1},
there exist a ball $B'\subset B$ centered at a point of $F=\d \om_1 \cap \d \om_2$ and a constant $c_1$  (depending on $\ve'$ and the CDC constants) for which $x_1$ and  $x_2$ are $c_{1}$-corkscrew points for $B'$ in $\om_1$ and $\om_2$ respectively, and so that \eqref{eq:Amu1} is satisfied with slightly different constants.  If $x_{1}$ was already a $c$-corkscrew point for $B$, then $x_{2}$ is a corkscrew point for $B$ as well with a constant depending on $c$ and the constants in Theorem \ref{teo1}.
\end{lemma}

\begin{proof}
Let us first record that by taking complements in \eqref{eq:Amu1}, we know that for $E \subset B$,
\begin{equation}\label{eq:Amu2}
\mbox{if}\quad \omega_2^{x_2}(E)\leq (1-\ve')\, \omega_2^{x_2}(B), \quad\text{ then }\quad \omega_1^{x_1}(E)\leq (1-\ve)\,\omega_1^{x_1}(B).
\end{equation} 

We fix a ball $B$ centered at $F= \d \Omega_1\cap \d \Omega_2$ with radius $r$, and for $i=1,2$, we let $x_i \in \tfrac{1}{4} B \cap \Omega_i$ be the points given in Theorem \ref{teo1} and  $\xi_i \in  \d \Omega_i$  be the points for which $|x_i - \xi_i|=\delta_i(x_i)=:\dist(x_i, \d \om_i)$. Notice that since $x_{i}\in \frac{1}{4} B$, we know $\delta_{i}(x_{i})<\frac{r(B)}{4}$, and so $\xi_i  \in \tfrac{1}{2} B \cap \d \om_i$. Let $\eta \ll 1$ to be chosen momentarily (depending on the constants in \eqref{eq:hm-holder}, the constant $c$ in \eqref{eq:bourgain}, $\ve$ and $\ve'$). 

Note that if $\min\{\delta_{1}(x_{1}),\delta_{2}(x_{2})\}\geq \eta r(B)$, then there is nothing to show as $x_{1}$ and $x_{2}$ are already corkscrews with constant $\eta$, and so we will assume
\begin{equation}\label{eqeta1}
\min\{\delta_{1}(x_{1}),\delta_{2}(x_{2})\}<\eta\, r(B).
\end{equation}

{\bf Claim:} There are $\tau,M>0$ so that, if  $\eta\ll \tau/M$ (with $M$ and $\tau$ depending also on
$\ve$ and $\ve'$), then
\begin{equation}
\label{e:tm}
\tau\delta_{1}(x_{1})\leq \delta_{2}(x_{2})\leq M \delta_{1}(x_{1}).
\end{equation}

Suppose that {$0<\tau< (c\,C^{-1}\,\ve )^{\frac{1}{n-1}}/2$, where $C$ is the constant in Lemma \ref{lem:u-decay-infinty}}. Note that since $\xi_{2}\in \frac{1}{2} B$ and $\delta_{2}(x_{2})<\frac{r(B)}{4}$, $B(\xi_{2},2\delta_2(x_{2}))\subset B$. If $\delta_{2}(x_{2})< \tau \delta_{1}(x_{1})\leq \tau |x_{1}-\xi_{2}|$, then, {as  $\omega_{1}^{x_{1}}(B(\xi_{2},2\delta_{2}(x_{2})))$ vanishes on $\d\om \setminus B(\xi_{2},2\delta_{2}(x_{2}))$ (since $\tau \ll 1$),}  by Lemma \ref{lem:u-decay-infinty},
\[
\omega_{1}^{x_{1}}(B(\xi_{2},2\delta_{2}(x_{2})))
\leq C \frac{(2\delta_{2}(x_{2}))^{n-1}}{|x_{1}-\xi_{2}|^{n-1}}
\leq C \frac{(2\tau \delta_{1}(x_{1}))^{n-1}}{|x_{1}-\xi_{2}|^{n-1}}
\leq C (2\tau)^{n-1}<\ve c
\leq \ve  \omega_{1}^{x_{1}}(B).
\]
Thus, by Lemmas \ref{lem:bourgain} and \ref{eq:Amu1},
\[
c\leq \omega_{2}^{x_{2}}(B(\xi_{2},2\delta_{2}(x_{2})))
\leq \ve' \omega_{2}^{x_{2}}(B),\]
which is a contradiction for $\ve'<c$. Therefore, $\delta_{2}(x_{2})> \tau \delta_{1}(x_{1})$.  In particular, 
\[
\delta_{1}(x_{1})
=\min\{\delta_{1}(x_{1}),\tau^{-1} \delta_{2}(x_{2})\}
\leq \tau^{-1} \min\{\delta_{1}(x_{1}), \delta_{2}(x_{2})\}
\leq \frac{\eta}{\tau}r(B).
\]
Note that if $x_{1}$ were already a $c$-corkscrew point for $B$, we would obtain a contradiction already, implying \eqref{e:tm}, which would then imply $x_{2}$ was a corkscrew point in $B$ with a different constant, and we would be done. Hence, we will assume now $x_{1}$ was not already a corkscrew point.

Let $M>0$ and assume that $\delta_{2}(x_{2})\geq M\delta_{1}(x_{1})$. If we pick $\eta^{-1}\gg \frac{4M}{\tau}$, then since $\xi_{1}\in \frac{1}{2} B$, we have that
\[
B(\xi_{1}, M\delta_{1}(x_{1}))
\subset B(\xi_{1},\tfrac{M \eta}{\tau}r(B)) 
\subset B(\xi_{1},\tfrac{r(B)}{4})\subset B.\]
Thus, for $M$ large enough, {by Lemma \ref{lem:Holder-hm}},
\begin{align}
\label{e:w2M}
\omega_{1}^{x_{1}}(B\backslash B(\xi_{1},M^{\frac{1}{2}}\delta_{1}(x_{1})))
&\leq C\ps{\frac{\delta_{1}(x_{1})}{M^{\frac{1}{2}}\delta_{1}(x_{1})}}^{\alpha}
 \leq \frac{C\omega_{1}^{x_{1}}(B)}{c\,M^{\alpha/2}}<\ve \omega_{1}^{x_{1}}(B).
\end{align}
Thus, by \eqref{eq:Amu1},
\[
\omega_{2}^{x_{2}}(B\backslash B(\xi_{1},M^{\frac{1}{2}}\delta(x_{1})))
<\ve' \omega_{2}^{x_{2}}(B)\leq \ve'.
\]
However,  since
\[
|x_{2}-\xi_{1}|\geq \delta_{2}(x_{2})\geq M\delta_{1}(x_{1}),\]
{ it is clear that $\omega_{2}^{x_{2}}(B(\xi_{1},M^{\frac{1}{2}}\delta_{1}(x_{1})))$ vanishes on $\d \om \setminus B(\xi_{1},M^{\frac{1}{2}}\delta_{1}(x_{1}))$. So, if $M$ is sufficiently large, by Lemma \ref{lem:u-decay-infinty}, we have that }
\begin{align*}
\omega_{2}^{x_{2}}(B(\xi_{1},M^{\frac{1}{2}}\delta_{1}(x_{1})))
& \leq C \ps{\frac{M^{\frac{1}{2}}\delta_{1}(x_{1})}{|x_{2}-\xi_{1}|}}^{n-1}
\leq  C \ps{\frac{M^{\frac{1}{2}}\delta_{1}(x_{1})}{M\delta_{1}(x_{1})}}^{n-1}\\
&= C M^{\frac{1-n}{2}}<\frac{1-\ve'}{c} \leq (1-\ve')\omega_{2}^{x_{2}}(B),
\end{align*}
which contradicts \eqref{e:w2M}. Thus, \eqref{e:tm} follows. 

We claim now that, for some small $\rho>0$ depending on $\ve$ and $\ve'$, $\max\{\delta_{1}(x_{1}),\delta_{2}(x_{2})\} \geq \rho\, |\xi_1 - \xi_2|$. Indeed, assume that
$$\max\{\delta_{1}(x_{1}),\delta_{2}(x_{2})\} < \rho\, |\xi_1 - \xi_2|=:  2\rho\,d .$$
We denote $B^i:= B(\xi_i, d)$.
Then, by Lemma \ref{lem:Holder-hm}, for $\rho$ small enough,
\begin{equation}\label{eq:c11}
\omega_{i}^{x_{i}}(B\backslash B^{i})
 \leq C\ps{\frac{|x_{i}-\xi_{i}|}{d}}^{\alpha}
<C \ps{\frac{2\rho d}{d}}^{\alpha}
\leq \frac{C (2\rho)^{\alpha}}{c}\omega_{i}^{x_{i}}(B)
<\ve\omega_{i}^{x_{i}}(B).
\end{equation}
So if we apply \eqref{eq:Amu2} using \eqref{eq:c11} for $i=1$, we infer that
\begin{equation}\label{eq:c12}
\omega_2^{x_2}(B^1) \geq (1- \ve')\, \omega_2^{x_2}(B) \geq c\, (1- \ve').
\end{equation}
Since $d=|\xi_{1}-\xi_{2}|/2$, we know that $B^1 \cap B^2 = \varnothing$ and thus, by \eqref{eq:c11},  and \eqref{eq:c12},
$$
1= \| \omega_2^{x_2} \| \geq \omega_2^{x_2}(B^2) + \omega_2^{x_2}(B^1) \geq (1 - C \rho^\alpha) +  c \,(1- \ve') >1.
$$  
This is a contradiction and thus,  
\begin{equation}\label{eq:rho-delta1-delta2}
\rho\, |\xi_1 - \xi_2| 
\leq \max\{\delta_{1}(x_{1}),\delta_{2}(x_{2})\}   
\leq \frac{M}{\tau} \min\{\delta_{1}(x_{1}),\delta_{2}(x_{2})\}.
\end{equation}

Set now $m:=\max\{\delta_{1}(x_{1}), \delta_{2}(x_{2})\}$ and define $ \wt B=B \left(\xi_{1}, 2 \rho^{-1} m \right)$. If we choose the constants so that $ \eta\leq \frac{\rho\tau}{40 M}$ , then, since $\xi_{1}\in \frac{1}{2} B$,  it holds that
\[
B^1 \subset \wt B\subset 10 \wt B \subset 
B\left(\xi_{1},\tfrac{20 M}{\rho\tau} \min\{\delta_{i}(x_{i})\}\right)
\subset B\left(\xi_{1},\tfrac{20 M\eta}{\rho\tau}r(B)\right)
\subset B\left(\xi_{1},\tfrac{r(B)}{2}\right)\subset B.
\]
Also, by definition of $m$ and taking $\rho\leq 1/2$,
\[
B(x_{1},\delta_{1}(x_{1}))
\subset B(\xi_{1},2\delta_{1}(x_{1}))
\subset \wt B,\]
while by \eqref{eq:rho-delta1-delta2},
\[
B(x_{2},\delta_{1}(x_{1}))
\subset B(\xi_{2},\delta_{1}(x_{1})+\delta_{2}(x_{2}))
\subset B(\xi_{1}, (\rho^{-1}+2)m )
\subset \wt B.
\]
Therefore, by \eqref{e:tm} and the above considerations, $x_{1}$ and $x_{2}$ are corkscrew points for $\wt B$ contained in $\wt B$.

Recall now that we denoted $F= \d \om_1 \cap \d \om_2$. We will now show that there exists $\xi \in \wt B \cap F$. Assume on the contrary that this is not the case. 
Then, in light of \eqref{eq:c12} and  $B^1 \subset \wt B$, we  have that $\hm_2^{x_2}(\wt B ) \geq (1- \ve')\hm_2^{x_2}(B ) $.
Thus,
$$\hm_2^{x_2}(B \setminus F) \geq \hm_2^{x_2}(\wt B \setminus F)=  \hm_2^{x_2}(\wt B ) \geq (1- \ve') \hm_2^{x_2}(B).$$
Notice also that if $K=B \cap (\d \om_2 \setminus \d \om_1) \subset B$, then it is clear that $\hm_1^{x_1}(K) =0$, which, by \eqref{eq:Amu1}, implies that 
$\hm_2^{x_2}(K) \leq \ve' \hm_2^{x_2}(B)$. Therefore, we have that
$$ \hm_2^{x_2}(B \cap F) \geq (1- \ve') \hm_2^{x_2}(B).$$
Combining the latter two inequalities we reach a contradiction for $\ve'< \frac{1}{2}$. Therefore, there exists $\xi \in F \cap \wt B$ and if we set $r= 8 r(\wt B)$, then $x_{1}$ and $x_{2}$ are corkscrew points for $B(\xi, r)$ and are contained in $\wt B \subset \tfrac{1}{4}B(\xi, r) \subset \frac{5}{2} \wt B$. 
Next we verify \eqref{eq:Amu1} still holds for $B(\xi, r)$ in place of $B$.  Indeed, given $E \subset B(\xi, r)$ such that 
$$\omega_1^{x_1}(E)\leq \ve\, \omega_1^{x_1}(B(\xi, r)) \leq \ve\, \omega_1^{x_1}(B),$$
then by  \eqref{eq:Amu1} and that $4 \wt B \subset B(\xi,r)$, 
$$\omega_2^{x_2}(E)\leq \ve'\,\omega_2^{x_2}(B)\leq \frac{\ve'}{1-\ve'}\,\omega_2^{x_2}(\wt B) \leq \frac{\ve'}{1-\ve'}\,\omega_2^{x_2}(B(\xi, r)).$$
Since $\ve'$ is sufficiently small, $\tfrac{\ve'}{1-\ve'}$ is sufficiently small as well.

\end{proof}
\vvv

\vvv

\subsection{Reduction to flat balls far away from the poles}

Before we state our lemma we need to introduce some notation. 

For a measure $\mu$, $\xi\in \supp \mu$, $L$ an $n$-plane, and $r>0$, we define
\[
\beta_{\mu,1}^{L}(\xi,r)=\frac{1}{r^{n}}\int_{B(\xi,r)} \frac{\dist(x,L)}{r}d\mu(x)\]
and 
\[
\beta^n_{\mu,1}(\xi,r)=\inf_{L}\beta_{\mu,1}^{L}(\xi,r),\]
where the infimum is over all $n$-dimensional planes $L$. 
Recall also that we have defined
$$\Theta^n_\mu(x,r)= \frac{\mu(B(x,r))}{r^n}.$$

To shorten notation, from now on we will also write $\omega_i$ instead of $\omega_i^{x_i}$, for $i=1,2$.
\vv
\begin{lemma}
\label{l:beta-reduction}
Under the assumptions of Theorem \ref{teo1}, suppose also that $x_{i}$ is a corkscrew point in $\Omega_{i}$ for the ball $B$.
 Then for any $\eta,\tau>0$, we may find a ball $B'$ contained in $B$ for which $  r(B)\lec_{\eta,\tau} r(B')\leq   \tau\, r(B)$, $\omega_{1}(B')\gec_{\tau,\eta} \omega_{1}(B)$, $\omega_{2}(B')\gec_{\tau,\eta} \omega_{2}(B)$, and $\beta_{\omega_{2},1}^{n}(B')<\eta\, \Theta^n_{\omega_{2}}(B')$. 
\end{lemma}

\begin{proof}
Recall that, by Remark \ref{remdob} we have $\omega_i^{x_i}(\tfrac23 B) \sim \omega_i^{x_i}(B)$.
Then, by Lemma \ref{l:nullboundary}, we can replace $B$ with a slightly concentric smaller ball for which $\omega_{1}(\d B)=\omega_{2}(\d B)=0$ 
and 
\begin{equation}\label{eqdob22}
\omega_i^{x_i}(\tfrac34 B) \sim \omega_i^{x_i}(B)\sim 1\quad \mbox{for $i=1,2$},
\end{equation}
and we will still have $\omega_{2}\in \textup{weak-}\wt A_{\infty}(\omega_{1},B)$ with slightly different constants. Since our balls are open, we can choose this ball close enough to the original so we still have $x_{i}\in \frac{1}{4}B$. Without loss of generality, we may assume $B =\bB$. 
 
 Suppose instead that the statement in the lemma is false. Thus, there exist { $\eta, \tau >0$}, and a sequence of CDC domains $\om_i^{j}$ ($i=1,2$) so that \\
(a) $0 \in \d \om_i^{j}$; \\
(b) There exist corkscrew points $x_i^{j} \in \om_i^j \cap \frac{1}{4}\bB$ so that if $\omega_{i}^{j} = \omega_{\Omega_{i}^{j}}^{x_{i}^{j}}$ then $\omega_{2}^{j}  \in$ weak-$\wt A_{\infty}(\omega_{1}^{j}, \bB)$;\\
(c) For every open ball $B \subset \bB$ centered on $\d \om_1^{j} \cap \d \om_2^{j}$ with $j^{-1}< r(B) \leq \tau/10$ at least one of the following holds:
\begin{enumerate}
\item $\omega_{1}^{j}(B)< j^{-1} \,\omega_{1}^{j}(\bB)$, or
\item $\omega_{2}^{j}(B)< j^{-1} \,\omega_{2}^{j}(\bB)$, or
\item $\beta_{\omega_{2}^{j},1}^{n}(B)\geq \eta\, \Theta_{\omega_{2}^{j}}^{n}(B)$.
\end{enumerate}

%
It is not hard to see that we may pass to a subsequence so that $x_{i}^{j} $ converges to another corkscrew point $x_{i}^0 \in \om^j_i$, for all $j$. Thus, by Harnack's inequality, it is enough to assume that the latter cases hold if we replace $x_i^j $ with $x_{i}^{0}$.
This and Lemma  \ref{limlem} imply that, by passing to a subsequence, there exists a CDC domain $\om^\infty_i$ so that $\omega_{i}^{j}\warrow \omega_{i}^{\infty}=: \omega_{\Omega_{i}^{\infty}}^{x_{i}^{0}}$, with $\omega_{i}^{\infty}(\frac56 \bB)\sim \omega_{i}^{\infty}(\bB)$, by
\rf{eqdob22}.
Therefore, for all open balls $B\subset \bB$ centered on $\bB \cap \d \om^\infty_1 \cap \d \om^\infty_2$, with 
 $\omega_i^\infty(\partial B)=0$ and
$r(B)<\tau/10$, one of the following happens:
\begin{enumerate}
\item[(1')] $\omega_{1}^{\infty}(B)=0$, or
\item[(2')] $\omega_{2}^{\infty}(B)=0$, or
\item[(3')] $\beta_{\omega_{2}^{\infty},1}^{n}(B)\geq \eta\, \Theta_{\omega_{2}^{\infty}}^{n}(B)$.
\end{enumerate}

Note that by Lemma \ref{l:afinlim}, $\omega_{2}^{\infty}\in \textup{weak-}\wt A_{\infty}(\omega_{1}^{\infty},\bB)$, and by Lemma \ref{l:goodset}, there is $G\subset \bB \cap \d \om^\infty_1 \cap \d \om^\infty_2$  such that 
\begin{equation}
\label{e:G}
\omega_{2}^{\infty}(\bB\backslash G)<2\ve' \quad\mbox{and}\quad\omega_{1}^{\infty}|_{G}\ll\omega_{2}^{\infty}|_{G}\ll\omega_{1}^{\infty}|_{G}.
\end{equation}
By the main results of \cite{AMTV16}, $\omega_{2}^{\infty}|_{G}\ll \HH^{n}|_{G}$ and $G$ is $n$-rectifiable. Thus,
$$\lim_{r \to 0}\Theta^n_{\omega_{2}^{\infty}}(x,r) \in(0, \infty) \,\,\textup{for}\,\,\omega_{2}^{\infty}\textup{-a.e.}\,\, x\in G.$$
Moreover, { it was proved in \cite{AMTV16} that } $\beta_{\omega_{2}^{\infty},1}^{n}(x,r)/\Theta_{\omega_{2}^{\infty}}^{n}(x,r)\rightarrow 0$ as $r\rightarrow 0$ for $\omega_{2}^{\infty}$-a.e. $x\in G$, and so for such $x$, we must eventually have $\beta_{\omega_{2}^{\infty},1}^{n}(x,r)< \eta\, \Theta_{\omega_{2}^{\infty}}^{n}(B(x,r))$ for small enough $r$. But this means that for sufficiently small balls centered on $G\cap \frac56\bB$ (and thus contained in
$\bB$), neither (1'), (2'), or (3') can hold, and we obtain a contradiction.

\end{proof}

\begin{lemma}\label{lem:B'weakAinfty}
If $B' \subset B$ is the ball obtained in  Lemma \ref{l:beta-reduction}, then $\omega_{2}\in \textup{weak-}\wt A_{\infty}(\omega_{1},B')$ with slightly different constants.
\end{lemma}

\begin{proof}
Let $E \subset B'$ be such that $\omega_{1}(E) < \ve\,\omega_{1}(B') \leq \ve\,\omega_{1}(B)$. Then by \eqref{eq:Amu1} and Lemma \ref{l:beta-reduction}, we have that  there exists a constant $C$ depending on ${\eta, \tau}$, so that
$$\omega_{2}(E) < \ve'\,\omega_{2}(B) \leq C(\eta, \tau) \ve'\,\omega_{2}(B').$$
\end{proof}

\section{Boundedness of the Riesz transform on $G$}\label{sec:4}

For the sake of convenience, from now on we will be using the notation $B$ and $\ve'$ for the ball $B'$ and the constant $C(\eta, \tau) \ve'$ obtained in Lemma \ref{l:beta-reduction} and Lemma \ref{lem:B'weakAinfty} respectively. 

\vv

Recall now that, with above convention, for fixed $\eta,\tau>0$ (to be chosen later), if $c_B$ is the center and $r(B)$ is the radius of the ball $B$, we have that there exists a positive constant $c_2 < 1$ depending on $\tau$ and $\eta$ such that 
\begin{enumerate}
\item  $\omega_{i}(B) \geq c_2$, for $i=1,2$;
\item $r(B) \lec \tau\, |x_i - c_{B}|$, for $i=1,2$;
\item $\omega_{2}\in \textup{weak-}\wt A_{\infty}(\omega_{1},B)$;
\item $\beta_{1,\omega_{2}}^{n}(B)<\eta\, \Theta^n_{\omega_{2}}(B)$.
\end{enumerate}

Note that using (1), it follows that for $i=1,2$,
$$
P_{\gamma,\omega_{i}}(B)\leq \sum_{j\geq0} 2^{-j\gamma}\,\frac{1}{2^{jn} r(B)^n} \lesssim_{n,\gamma} c_2^{-1}\Theta^n_{\omega_{i}}(B),
$$
for all $\gamma \in [0,1]$.
\vv

In the next lemma we prove some useful estimates on the good set $G$ defined in Lemma~\ref{l:goodset}.

\begin{lemma}\label{lem:G-estimates}
If $G$ is the set of points in $B$ obtained in Lemma \ref{l:goodset} for $\mu=\hm_1$ and $\nu=\hm_2$, then for $i=1, 2$, there exist $C_3$ and $C_4$ positive constants depending on $\tau,\eta$ and $A$, such that
 \begin{align}
\Theta_{\omega_{i}}^{n}(x,r) &\leq C_3 \Theta_{\omega_{i}}^{n}(B)\quad \mbox{ for all }x\in G \mbox{ and } 0<r\leq 2r(B),\label{e:upperregular}\\
\RR_{*}(\chi_{2B}\omega_{i} )(x)&\leq C_4  \Theta^n_{\omega_{i}}(B)\quad\mbox{ for all }x\in G.\label{e:rieszbounded}
\end{align}
\end{lemma} 

\begin{proof}
If we combine Lemma \ref{l:otherside} and Theorem \ref{t:ACF}, along with \eqref{eq:mu-nu-comparable} and the fact that for some $\tau \ll 1$, it holds $r(B) \lec \tau\, |x_i - c_{B}|$, then we have that
\begin{equation*}
\Theta_{\omega_{i}}^{n}(x,r)\lec \Theta_{\omega_{i}}^{n}(B), \mbox{ \,\,for all }x\in G \mbox{ and } 0<r \leq 2r(B).
\end{equation*}

Finally, \eqref{e:rieszbounded} can be proven exactly as in  \cite[Lemma 6.3]{AMTV16}, since $B$ is a $P_{\omega_{i},\gamma}$-doubling ball. The proof of the lemma is now concluded.
\end{proof}

\begin{lemma}\label{lem:applic-Tb-suppressed}
If $G \subset B$ is the set in Lemma \ref{lem:G-estimates}, then 
\begin{equation}\label{eq:w1G}
\hm_2(G) > (1-2\ve')\, \hm_2(B) \quad  \textup{and}\quad \hm_1(G) >  \frac{c_2 (1-2\ve')}{A} \, \hm_1(B),
\end{equation}
 and for $i=1,2$  it holds
\begin{equation}\label{w2G-pol-growth} 
{\hm_i}|_{G}(B(x,r)) \lesssim   \Theta_{\hm_i}^n(B)\, r^n\quad \textup{for all}\,\, x \in \R^{n+1}\,\, \textup{and all}\,\, r>0,
\end{equation}
and
\begin{align}\label{eq:Riesz-bound-G}
\| \RR_{{\hm_i}|_{G}}  &\|_{L^2({\hm_i}|_{G}) \to L^2({\hm_i}|_{G})} \lesssim  \Theta_{\hm_i}^n(B). \end{align}
\end{lemma}

\begin{proof}
The first equation of \eqref{eq:w1G} follows from the application of Lemma \ref{l:goodset} for $\mu=\hm_1$ and $\nu=\hm_2$. Then from the same lemma, since $A^{-1} \leq \frac{d \hm_1}{d \hm_2} \leq A$ and $\hm_2(B)>c_2$, we obtain the second inequality from the first one. 
The estimate \eqref{w2G-pol-growth} is an immediate consequence of \rf{e:upperregular}.
Lastly, using \eqref{e:rieszbounded}, once can prove  \eqref{eq:Riesz-bound-G} exactly as in the proof of Theorem 3.3 in \cite[pp. 8-10]{AMT16}. 
\end{proof}

 
 \vv

The next theorem is essentially shown in \cite[Theorem 2.1]{NToV14-pubmat}, although it is not stated as such there.
\begin{theorem}\label{t:pajot}
There is $C>0$ depending only on $n$ such that for all $p,q>1$ the following holds. Let $\nu$ be a measure supported in a ball $B$ such that $\nu(B(x,r))\leq r^{n}$ for all $x\in \R^{n+1}$ and $r>0$, and define
 \begin{align*}
F_p& =\{ x \in \supp(\nu): \mbox{ for }0<r\leq\diam \supp(\nu), \;\; \nu(B(x,r)) \geq r^{n}/p \},\\
F_{p,q}&=\{ x\in F_{p}: \mbox{ for }0<r\leq \diam \supp(\nu), \;\; \nu(B(x,r)\cap F_{p})\geq \frac{r^{n}}{pq}\}.
\end{align*}
 If $\RR_\nu:L^{2}(\nu)\rightarrow L^{2}(\nu)$ is bounded, then there is a $\frac{C}{pq}$-AD regular measure $\sigma$ so that $\sigma|_{F_{p,q}} = \nu|_{F_{p,q}} $ and $\RR_\sigma:L^{2}(\sigma)\rightarrow L^{2}(\sigma)$ is bounded.
\end{theorem}

The following lemma shows that $F_{p,q}$ captures a big piece of $F_p$ in $\nu$-measure.
\begin{lemma}\label{lem:nu-BP-ADR}
Let $\nu$, $F_p$ and $F_{p,q}$ be as in Theorem \ref{t:pajot}. If $\nu(F_p) \geq \delta \nu(B)$ for some $\delta \in (0, \tfrac{1}{2})$, then there exists $q=q(n,p, \delta)$ such that $\nu(F_{p,q})\geq \tfrac{1}{2} \nu(F_{p})$.
\end{lemma}

\begin{proof}
If $x\in F_{p}\backslash F_{p,q}$, we let 
\[
r_{x} = \sup\left\{ r>0: \nu(B(x,r)\cap F_{p})< \frac{r^{n}}{pq}\right\}.
\]
Then we clearly have
\[
\nu(B(x,r_{x})\cap F_{p}) \leq \frac{r_{x}^{n}}{pq}.\]
 If $\{B_{j}\}_j$ is a Besicovitch subcovering from the collection $\{B(x,r_{x})\}$, then
\begin{multline*}
\nu(F_{p}\backslash F_{p,q})
\leq \sum_{j} \nu(B_{j}\cap F_{p})
< \frac{1}{pq}\sum_{j} r(B_{j})^{n} \\
\leq \frac{1}{q} \sum_{j} \nu(B_{j})
\lec \frac{1}{q} \,\nu\biggl(\bigcup_j B_{j}\biggr)
\lec \frac{1}{q}\, \nu(B)
\lec \frac{1}{q\delta} \,\nu(F_{p}).
\end{multline*}
Thus, for $q$ large enough, $\nu(F_{p}\backslash F_{p,q})<\frac{1}{2} \nu(F_{p})$, and so 
\[
\nu(F_{p,q})
= \nu(F_{p})-\nu(F_p\setminus F_{p,q})
\geq \frac{1}{2}\, \nu(F_{p}). \]
\end{proof}

\begin{rem}\label{rempq}
From the arguments above it is clear that, given any constant $\kappa>0$, if the constant $q$ in the 
lemma above can be chosen so that $\nu(F_{p,q})\geq (1-\kappa) \nu(F_{p})$.
\end{rem}

For fixed $\rho \in (0,1)$ to be chosen later, we define
$$
G^{bd} = \{x\in G: \Theta_{\omega_{1}}^{n}(x,r) \geq \rho\,\Theta_{\omega_{1}}^{n}(B) \mbox{ for all } r\in (0, 2r(B)]\},\;\;\;\; G^{sd}=G\backslash G^{bd},
$$
where $bd$ and $sd$ stand for ``big density" and ``small density" respectively.

 \vv
 
\section{Case 1: $G^{bd}$  is a big piece of $B$ in harmonic measure}\label{sec:5}

Fix now $\delta \in (0,1)$ to be chosen later and assume that 
\begin{equation}\label{eq:nu-Gbd-big}
\hm_1(G^{bd})\geq \delta \,\hm_1(B).
\end{equation}

{
Let {$\rho'>0$ to be chosen} and 
\[
G'=\{x\in G^{bd}: \Theta_{\omega_{1}|_{G}}^{n}(x,r)\geq \rho'\Theta_{\omega_{1}}^{n}(x,r) \mbox{ for all } r\in (0,2r(B))\}.\]
Let $B_{j}$ be a Besicovitch covering from the collection of balls 
\[
\{B(x,r): x\in G^{bd}\backslash G', \;\; \Theta_{\omega_{1}|_{G}}^{n}(x,r)< \rho'\Theta_{\omega_{1}}^{n}(x,r), 0<r<2r(B)\}.
\]
Then 
\[
\omega_{1}(G^{bd}\backslash G')
\leq \sum_j \omega_{1}(B_{j}\cap G)
\leq \rho' \sum_j\omega_{1}(B_{j})
\lec_{n} \rho' \omega_{1}\biggl(\bigcup_j B_{j}\biggr)
\leq \rho' 
\lec \rho'\delta^{-1}  \omega_{1}(G^{bd}).
\]
For $\rho'\ll \delta$, this implies 
\[
\omega_{1}( G')
\geq \frac{1}{2}\omega_{1}(G^{bd})
\geq \frac{\delta}{2}\omega_{1}(B).
\]
Also, for $x\in G'$ and $0<r<r(B)$,
\begin{equation}
\label{e:lowerregular1}
\Theta_{\omega_{1}|_{G}}^{n}(x,r)
\geq \rho' \Theta_{\omega_{1}}^{n}(x,r)
\geq \rho'\rho\, \Theta_{\omega_{1}}^{n}(B)
\geq \rho'\rho \,\Theta_{\omega_{1}|_{G}}^{n}(B).
\end{equation}


Set $\nu= \frac{{\hm_1}|_{G}}{C\Theta^n_{{\hm_1}|_G}(B)}$ , where $C>1$ is the implicit constant in \eqref{w2G-pol-growth},
and $F_p$ is the set defined in Theorem \ref{t:pajot}. If { we set $p = C (\rho'\rho)^{-1}$, $F_p \supset  G'$ (since $\rho<1$)}, by Lemma \ref{lem:applic-Tb-suppressed},   we can apply Theorem \ref{t:pajot} to infer that there is a $\frac{\rho \rho'}{Cq}$-AD-regular measure $\sigma$ so that $\sigma|_{F_{p,q}} = \nu|_{F_{p,q}} $ and $\RR_\sigma$ is bounded in $L^{2}(\sigma)$. Moreover, by the fact that $F_{p,q} \subset G^{bd} \subset G \subset B$ {(as $q>1$ and $\rho'\ll 1$)}, inequality \eqref{eq:nu-Gbd-big} and Lemma \ref{lem:nu-BP-ADR}, we have that
\begin{equation}\label{eq:Fpqbigw1}
\nu(F_{p,q}) \geq \frac{1}{2} \nu(F_{p})
 \geq  \frac{1}{2} \nu(G')
 =\frac{1}{2}  \frac{{\hm_1}(G')}{C\Theta^n_{{\hm_1}|_G}(B)} 
  \geq  \frac{\delta}{4} \frac{{\hm_1}(B)}{C\Theta^n_{{\hm_1}|_G}(B)}  \geq  \frac{\delta}{4} \nu(B).
\end{equation}

By the main results of \cite{NTV14-Acta}, $\sigma$ is a uniformly rectifiable measure with support $\Sigma_B$, and thus, by \eqref{eq:Fpqbigw1}, we obtain
\[
\omega_{1}(\Sigma_B\cap F\cap B)
\geq 
\omega_{1}( F_{p,q}) 
\geq  \frac{\delta}{4} \omega_{1}(B),\]
as desired.}

Also note that if $x_{i}$ were already corkscrew points in $\frac{1}{4}B$ and we did not need to apply Lemma \ref{lem:B'weakAinfty}, then \eqref{e:bigpiece} follows immediately from \eqref{eq:Fpqbigw1}. 

\vv
\section{Case 2: $G^{sd}$  is a very big piece of $B$ in harmonic measure}\label{sec:6}
 
Now we assume instead that 
\begin{equation}
\label{e:case2}
\omega_{1}(G^{bd})< \delta \omega_{1}(B).
\end{equation}
For $\delta<\ve$, by the weak-$A_{\infty}$ property, we also know
\begin{equation}
\label{e:case22}
\omega_{2}(G^{bd})< \ve' \omega_{2}(B).
\end{equation}
Further, by Lemma \ref{l:goodset}, we  have  
\begin{equation}
\label{e:case22bis}
\omega_{2}(B \setminus G) \leq 2 \ve' \omega_{2}(B),
\end{equation}
and then, combining \eqref{e:case22} and \eqref{e:case22bis}, since $G^{bd}=G \backslash G^{sd}$, we conclude that
\begin{equation}
\label{e:w2Gsd}
\omega_{2}(B \setminus G^{sd}) \leq 3 \ve' \omega_{2}(B).
\end{equation}

Recall now that $\eta,\tau>0$ are small constants to be chosen later and we have that
\begin{align}
P_{\gamma,\omega_{2}}(B) &\lesssim \Theta^n_{\omega_{2}}(B),\,\,\textup{for any}\,\, \gamma \in [0,1],\label{e:B'-doub}\\
\beta_{\omega_{2},1}(B) &<\eta\,\Theta^n_{\omega_{2}}(B),\label{e:b'b eta} \\
r(B) &\lec \tau\, |x_2 - c_{B}|, \label{e:B-cork}
\end{align}
where $c_B$ stands for the center of the ball $B$.


\begin{lemma}
For $x\in G^{sd}$, 
\begin{equation}
\label{e:R-K}
|\RR\omega_{2}(x)-K(x-x_{2})|\lec  \rho \, \Theta^n_{\omega_{2}}(B),
\end{equation}
where $0< \rho \ll 1$ is the constant in the definition of $G^{bd}$ and the implicit constant depends on $n, a, A$ and the CDC constants.
\end{lemma}

\begin{proof}
This lemma is also essentially proven in  \cite[Lemma 6.5]{AMTV16}. We sketch its proof for the sake of completeness. Let $\vphi:\R^{n+1}\to[0,1]$ be a radial $C^\infty$ function  which vanishes on $B(0,1)$ and equals $1$ on $\R^{n+1}\setminus B(0,2)$,
and for $\ve>0$ and $z\in \R^{n+1}$ denote
$\vphi_\ve(z) = \vphi\left(\frac{z}\ve\right) $ and $\psi_\ve = 1-\vphi_\ve$.
We set
$$\wt\RR_\ve\omega_2(z) =\int K(z-y)\,\vphi_\ve(z-y)\,d\omega_2(y),$$
where $K(\cdot)$ is the kernel of the $n$-dimensional Riesz transform. For a fixed $x \in G^{sd}$ and $z\in \R^{n+1}\setminus \bigl[\supp(\vphi_r(x-\cdot)\,\omega_2)\cup \{x_2\}\bigr]$, consider the function
$$v_r(z) = \EE(z-x_2) - \int \EE(z-y)\,\vphi_r(x-y)\,d\omega_2(y),$$
so that, by Remark 3.2 from \cite{AHM3TV}, { if $m$ stands for the $(n+1)$-dimensional Lebesgue measure, it holds}
\begin{equation}\label{eqfj33}
 G_{\Omega_2}(z,x_2) = v_r(z) - \int \EE(z-y)\,\psi_r(x-y)\,d\omega_2(y)\quad \mbox{ for $m$-a.e.   $z\in\R^{n+1}$.}
\end{equation}
By \eqref{eqker} we have
$$\nabla v_r(z) = c_n\,K(z-x_2) - c_n\,\RR(\vphi_{ r}(\cdot-x)\,\omega_2)(z).$$
In the particular case that $z=x$ we get
$$\nabla v_r(x) = c_n\,K(x-x_2) - c_n\,\wt\RR_r\omega_2(x).$$

Let $r_{j}\rightarrow 0$ be a sequence of radii so that $B_{j}=B(x,r_{j})$ is $2a$-$P_{\omega_1}$-doubling for every $j$ and let $r_x$ be the largest ball for which $\frac{\omega_2(B(x, r_x))}{r_x^n} < \rho\, \Theta^n_{\omega_{2}}(B)$. Then, following the proof of \cite[Lemma 6.5]{AMTV16}, we can  prove that for $r_j \leq r_x/8$,
$$|\nabla v_{r_j}(x)| \lesssim \frac{\omega_2(B(x, 4 r_j))}{r_j^n} \lesssim a \frac{\omega_2(B(x, r_j))}{r_j^n} \lesssim   \frac{\omega_2(B(x, r_x))}{r_x^n} \lesssim_{a,A}  \rho\,\Theta^n_{\omega_{2}}(B),$$
 where in the second inequality we used that $B_{j}=B(x,r_{j})$ is $2a$-$P_{\omega_2}$-doubling, in the third one Lemma \ref{l:otherside}, and in the last one the definition of $r_x$. To conclude the lemma it suffices to take $j \to \infty$.
\end{proof}

The next lemma was also essentially proven in  \cite[Lemma 6.5]{AMTV16}. 
\begin{lemma}\label{lem:oscilRiesz}
It holds
\begin{equation}
\label{e:GT3a}
\int_{G^{sd}} \left|\RR\omega_{2}(x) - \frac{1}{\omega_{2}(G^{sd})} \int_{G^{sd}} \RR\omega_{2} \, d\omega_2\right|^2\,d\omega_{2}(x) \lec (\rho+ \tau) \,\Theta^n_{\omega_{2}}(B)^2 \omega_{2}(B).
\end{equation}
\end{lemma}
\begin{proof}
If $x \in G^{sd} \subset B$, 
\begin{align*}
\left|\RR\omega_{2}(x) - \frac{1}{\omega_{2}(G^{sd})} \int_{G^{sd}} \RR\omega_{2} \right| &\stackrel{\eqref{e:R-K}}{\lec}
\rho\, \Theta^n_{\omega_{2}}(B) + |K(x-x_{2})- m_{\omega_{2},G^{sd}}(K(\cdot - x_{2}))| \\
&\leq \rho\, \Theta^n_{\omega_{2}}(B) + \sup_{y \in G^{sd}} |K(x-x_{2})- K(y- x_{2})| \\
&\lec  \rho\, \Theta^n_{\omega_{2}}(B)  +{\omega_{2}}(B) \, \frac{r_{B}}{|x_{2}-c_{B}|^{n+1}} \\
&\stackrel{\eqref{e:B-cork}}{\lec}  \rho\, \Theta^n_{\omega_{2}}(B) + \tau\,\Theta^n_{\omega_{2}}( B)=( \rho+\tau ) \Theta^n_{\omega_{2}}(B),
\end{align*}
where we also used that $\omega_2(B) \approx 1$. 

\end{proof}

Now we conclude the proof. We wish to apply the following theorem which is a corollary of the main theorem in \cite{GT}. Its proof can be found in \cite[Section 3]{AMT16}.

\begin{theorem} \label{t:GT}
Let $\mu$ be a Radon measure in $\R^{n+1}$ and $B\subset \R^{n+1}$ a ball with $\mu(B)>0$ so that the following conditions
hold:
\begin{itemize}
\item[(a)] For some constant $C_5>0$, $P_\mu(B) \leq C_5\,\Theta^n_\mu(B)$.

\item[(b)] There is some $n$-plane $L$ passing through the center of $B$ such that, for some constant $0<\delta_0 \ll 1$, $\beta_{\mu,1}^L(B)\leq \delta_0\,\Theta^n_\mu(B)$.

\item[(c)] For some constant $C_6>0$, there is $G_B\subset B$
such that
\begin{equation}
\label{e:GT1}
\sup_{0<r\leq 2 r(B)} \frac{\mu(B(x,r))}{r^n} + \RR_*(\chi_{2B}\,\mu)(x)\leq 
C_6\,\Theta^n_\mu(B)\quad \mbox{ for all $x\in G_B$}
\end{equation}
and
\begin{equation}
\label{e:GT2}
\mu(B\setminus G_B)\leq \delta_0 \,\mu(B).
\end{equation}
\item[(d)] For some constant $0<\tau_0\ll1$,
\begin{equation}
\label{e:GT3}
\int_{G_B} \left|\RR\mu(x) - \frac{1}{\mu(G_B)} \int_{G_B} \RR\mu \, d\mu\right|^2\,d\mu(x) \leq \tau_0 \,\Theta^n_\mu(B)^2\mu(B).
\end{equation}
\end{itemize}

Then there exists some constant $\theta>0$ such that if $\delta_0,\tau_0$ are small enough (depending on $C_5$ and $C_6$),
then there is a uniformly $n$-rectifiable set $\Gamma\subset\R^{n+1}$ such that
$$\mu(G_B\cap \Gamma)\geq \theta\,\mu(B).$$
The UR constants of $\Gamma$ depend on all the constants above.
\end{theorem}

We set $\mu = \omega_{2}$, and replace $G_{B}$ with $G^{sd}$. By \eqref{e:B'-doub} and \eqref{e:b'b eta} we have that (a) and (b) are satisfied for $B$ if we choose $\eta$ small enough. By \eqref{e:upperregular}, \eqref{e:rieszbounded}, and \eqref{e:w2Gsd}, one can see that (c) is satisfied if $\ve'$ is small enough. Finally, (d) follows from Lemma \ref{lem:oscilRiesz} if we choose $\tau \ll1$ and $\rho \ll 1$ small enough. The conclusion of Theorem \ref{t:GT} now holds for $\mu=\omega_{2}$, that is, there is a uniformly rectifiable set $\Sigma_B$ so that $\omega_{2}(\Sigma_B\cap G^{sd})\gec \omega_{2}(B)$. 

Recall that $\tfrac{d\omega_{1}}{d\omega_{2}} \sim 1$ on $G^{sd} \subset G$, and so we also have 
$$\omega_{1}(\Sigma_B\cap F \cap B) \geq \omega_{1}(\Sigma_B\cap G^{sd})\gec \omega_{1}(B).$$

Finally, if $x_{i}$ were corkscrew points for $\frac{1}{4}B$, \eqref{e:upperregular} and the previous inequality immediately imply that 

\[
\cH^{n}(\Sigma_B \cap G^{sd})
\gec \Theta^n_{\omega_{1}}(B)^{-1} \omega_{1}(\Sigma_B\cap G^{sd})\gec \Theta^n_{\omega_{1}}(B)^{-1}\omega_{1}(B)=r(B)^{n}.
\]


\vvv

\section{The two-phase problem in two-sided NTA domains}\label{sec:8}

This section is devoted to the proof of Theorem \ref{teoNTA}.
First we need to introduce some additional definitions to understand the statement of the theorem.

One says that a domain $\Omega\subset\R^{n+1}$ satisfies the Harnack chain condition if there are positive constants $c$ and $R$ such that
for every $\rho>0$, $\Lambda\geq 1$, and every pair of points
$x,y \in \Omega$ with $\dist(x,\partial\Omega),\,\dist(y,\partial\Omega) \geq\rho$ and $|x-y|<\Lambda\,\rho\leq R$, there is a chain of
open balls
$B_1,\dots,B_m \subset \Omega$, $m\leq C(\Lambda)$,
with $x\in B_1,\, y\in B_m,$ $B_k\cap B_{k+1}\neq \varnothing$
and $c^{-1}\diam B_k \leq \dist (B_k,\partial\Omega)\leq c\,\diam B_k.$  For $C\geq 2$, $\Omega$ is a {$C$-corkscrew domain} if for all $\xi\in \partial\Omega$ and $r\in(0,R)$ there are two balls of radius $r/C$ contained in $B(\xi,r)\cap \Omega$ and $B(\xi,r)\backslash \Omega$ respectively. If $B(x,r/C)\subset B(\xi,r)\cap \Omega$, we call $x$ a {\it corkscrew point} for the ball $B(\xi,r)$. Finally, we say that $\Omega$ is non-tangentially accessible (or just NTA)} if it satisfies the Harnack chain condition and it is a $C$-corkscrew domain. We say $\Omega_1$ is two-sided NTA if both $\Omega_1$ and $\Omega_2 = (\overline{\Omega_1})^{c}$ are NTA. 

NTA domains were introduced by Jerison and Kenig in \cite{JK}. In that work, among other results, they showed that harmonic measure $\omega$ is doubling in NTA domains, that is,
$$\omega(B(x,2r))\leq C\omega(B(x,r)) \quad\mbox{ for all $x\in\supp\omega$, $r>0$,}$$
 and also that $\supp\omega=\partial\Omega$.

Next we define what means that $\omega_2\in A_\infty(\omega_1)$ when $\Omega_1$ and $\Omega_2 = (\overline{\Omega_1})^{c}$ are two-sided NTA domains. This means that
 for every 
$\ve_0 \in (0,1)$
 there exists $\delta_0 \in (0,1)$ such that for every ball $B$ centered at $\partial\Omega$ and all
 $p_i\in\Omega_i\setminus 2B$, for $i=1,2$, the
 following holds: for any subset $F \subset B\cap\partial\Omega$,
 \begin{equation*}
\mbox{if}\quad\omega_1^{p_1}(F)\leq \delta_0\,\omega_1^{p_1}(B\cap\partial\Omega), \quad\text{ then }\quad \omega^{p_2}_2(F)\leq \ve_0\,\omega^{p_2}_2(2B).
\end{equation*}
Using the change of poles formula for NTA domains (see Lemma 4.11 in \cite{JK}), it is easy to check that the statement above is essentially independent of the precise poles $p_1,p_2$ that one chooses, as soon as they are far from 
$B\cap\partial\Omega_1$.


We say that a Radon measure $\mu$ has {\it big pieces of uniformly $n$-rectifiable measures} if there exists some $\ve>0$ such that, for every ball $B$ centered at $\supp\mu$ 
with radius at most $\diam( \supp\mu)$, there exists a uniformly $n$-rectifiable set $E$, with UR constants possibly depending on $\ve$, and a subset $G\subset E$ such that
\begin{equation}\label{equr00}
\mu(B\setminus G)\leq \ve\,\mu(B)
\end{equation}
and
\begin{equation}\label{equr01}
\mu(F)\approx_\ve \HH^n(F)\,\Theta_\mu^n(B)\quad \mbox{ for all $F\subset G$.}
\end{equation}
On the other hand, we say that 
 $\mu$ has {\it very big pieces of uniformly $n$-rectifiable measures} if, for every $\ve\in (0,1)$ exist sets $E$ and $G\subset E$ satisfying \rf{equr00} and 
\rf{equr01}.

It is well known that if $\mu$ is $n$-AD-regular, having big pieces or very big pieces of uniformly $n$-rectifiable measures is equivalent to the uniform $n$-rectifiability of $\mu$. Notice also that if $\mu$ is uniformly $n$-rectifiable and $w$ is an $A_\infty(\mu)$ weight, then the measure $w\mu$ has very pieces of uniformly $n$-rectifiable measures.

We say that $\Omega_1$ and $\Omega_2$ have {\it joint big pieces of chord-arc subdomains} if for any ball $B$ centered at $\partial\Omega$ with radius at most $\diam \partial\Omega$ there are two chord-arc domains
$\Omega_{B,i}\subset \Omega_i$, with $i=1,2$, such that 
$\HH^n(\partial\Omega_{B,1}\cap\partial\Omega_{B,2}\cap B)\gtrsim r(B)^n$.

To prove Theorem \ref{teoNTA} we will need to use 
the David-Christ cubes on a doubling space \cite{C,D} as generalized by Hyt\"onen and Martikainen \cite{HM2}. The precise result we need is the following.

\begin{theorem}
For $c_{0}<1/1000$, the following holds. Let $c_{1}=1/500$ and $\Sigma$ be a metric space. For each $k\in\bZ$ there is a collection $\cD_{k}$ of ``cubes,'' which are Borel subsets of $\Sigma$ such that
\begin{enumerate}
\item $\Sigma=\bigcup_{Q\in \cD_{k}}Q$ for every $k$,
\item if $Q,Q'\in \cD=\bigcup \cD_{k}$ and $Q\cap Q'\neq\varnothing$, then $Q\subseteq Q'$ or $Q'\subseteq Q$,
\item for $Q\in \cD_{k}$, there is {$\zeta_{Q}\in Q$} so that if $B_{Q}=B(\zeta_{Q},5c_{0}^{k})$, then
\[c_{1}B_{Q}\subseteq Q\subseteq B_{Q}.\]
\end{enumerate}
\label{t:}
\end{theorem}

For $Q\in \cD_{k}$, define $\ell(Q)=5c_{0}^{k}$, so that $B_{Q}=B(\zeta_{Q},\ell(Q))$. Note that for $Q\in \cD_{k}$ and $Q'\in \cD_{m}$, we have $\ell(Q)/\ell(Q')=c_{0}^{k-m}$. Given a measure $\mu$, we denote
$\Theta_\mu(Q) = \frac{\mu(Q)}{\ell(Q)^n}$.

\vv
\begin{proof}[Proof of Theorem \ref{teoNTA}]
(a) $\Rightarrow$ (b): This follows as a corollary of Theorem \ref{teo1} and an iteration
argument similar  to the one already appearing in \cite[Lemma 3.14]{Prats-Tolsa}.  Indeed, to prove that $\omega_1$ has very big pieces of uniformly $n$-rectifiable measures it suffices to show that, for every $\ve>0$, given some dyadic {David-Christ} cube $Q$
associated to the metric space $\partial\Omega_1$,
there exists a uniformly $n$-rectifiable set $E$, with UR constants possibly depending on $\ve$, and a subset $G\subset E$ such that
\begin{equation}\label{equr00*}
\omega_1(Q\setminus G)\leq \ve\,\omega_1(Q)
\end{equation}
and
\begin{equation}\label{equr01*}
\omega_1(F)\approx_\ve \HH^n(F)\,\Theta_{\omega_1}^n(Q)\quad \mbox{ for all $F\subset G$.}
\end{equation}
Further, by the change of pole formula for NTA domains we may assume that the pole $p_1\in\Omega_1$ for 
$\omega_1$ satisfies $\dist(p_1,\partial\Omega_1)\approx \dist(p_1,Q)\approx\ell(Q)$.

To find the required sets $E$ and $G$, for an integer $k\geq1$ and $0<\tau\ll1$, denote by $\LLD^k(Q)$ the family of maximal cubes { $P\in\DD$ }contained in $Q$ such that
$$\Theta^n_{\omega_1}(P)\leq \tau^k\,\Theta^n_{\omega_1}(Q).$$
Also set 
$$\LD^k(Q) = \bigcup_{P\in \LLD^k(Q)}P.$$
{ Obviously, the  definitions of $\LLD^k(Q)$ and $\LD^k(Q)$ depend also on $\tau$. However,  we will not reflect this in the notation (in a sense, one should consider that $\tau$ is a fixed parameter small enough).}
Using that $\omega_1$ is doubling, it is easy to check that $\Theta^n_{\omega_1}(P)\approx \tau^k\Theta^n_{\omega_1}(Q)$ for each $P\in\LLD^k(Q)$. Then, for $\tau$ small enough, we deduce that $\LD^k(Q)\subset \LD^{k-1}(Q)$.
We claim that, from Theorem \ref{teo1}, it is not difficult to deduce that
\begin{equation}\label{eq3812}
\omega_1(\LD^1(Q))\leq \delta\,\omega_1(Q)
\end{equation}
for some $\delta\in(0,1)$, using that a significant part of $\omega|_Q$ is supported on a uniformly $n$-rectifiable set.
For the reader's convenience we show the detailed arguments.
 
First note that Theorem \ref{teo1} ensures the existence of 
a uniformly $n$-rectifiable set $\Gamma\subset\R^{n+1}$ such that
$$\omega_1(Q\cap \Gamma)\geq \theta\,\omega_1(Q)$$
for some fixed $\theta>0$, with the UR constants of $\Gamma$ uniformly bounded. Now, let $I$ denote the subfamily of cubes from $\LLD^{1}(Q)$ which 
intersect $Q\cap \Gamma$. Consider a subfamily $J\subset I$ such that
\begin{itemize}
\item the balls $2 B_P$, $P\in J$, are pairwise disjoint, and
\item $\bigcup_{P'\in I} {P'}\subset \bigcup_{P\in J} 6 B_{P}.$ 
\end{itemize}
Then, using the fact that $\Theta_{\omega_1}^n(6 B_P)\approx\Theta_{\omega_1}^n(P)
\lesssim \tau \Theta_{\omega_1}^n(Q)$ for $P\in J$, we get
\begin{equation}\label{eqfif1}
\omega_1(\Gamma\cap \LD^{1}(Q)) \leq \sum_{P\in J} \omega_1(6 B_P) \lesssim
\tau\,\Theta_{\omega_1}(Q)
\sum_{P\in J} \ell(P)^n.
\end{equation}
By the $n$-AD-regularity of $\Gamma$ and the fact that $B_P\cap\Gamma \neq\varnothing$ for $P\in J$, we derive
$$\ell(P)^n\approx \HH^n(\Gamma\cap 2 B_P).$$
Thus, using the fact that the balls $2B_P$ are disjoint and contained in some fixed multiple of $B_Q$,  along with the $n$-AD-regularity of $\Gamma$, we get
$$\sum_{P\in J} \ell(P)^n \approx \sum_{P\in J}\HH^n(\Gamma\cap 2 B_P) \lesssim \HH^n(\Gamma\cap B_Q) \lesssim \ell(Q)^n.$$
Plugging this estimate into \rf{eqfif1} and choosing $\tau\ll\theta$, we obtain
$$\omega_1(\Gamma\cap \LD^{1}(Q))\leq C \tau\,\omega_1(Q) \leq \frac\theta2\,\omega_1(Q) \leq \frac12 \,\omega_1(Q\cap \Gamma).$$
Thus,
$$\omega_1(Q\cap \Gamma\setminus \LD^{1}(Q))\geq \frac12 \omega_1(Q\cap \Gamma)\geq \frac\theta2\,\omega_1(Q).$$
In particular, choosing $\delta=1-\frac12\theta$, we get $\omega_1(Q\setminus \LD^{1}(Q))\geq (1-\delta)\,\omega_1(Q)$ and proves the claim \rf{eq3812}.

Analogously to \rf{eq3812}, by the same arguments above  we deduce that, for each $P\in\LLD^{k-1}(Q)$,
$$\omega_1(P\cap\LD^k(Q))\leq \delta\omega_1(P).$$
Hence,
$$\omega_1(\LD^k(Q))\leq \delta \omega_1(\LD^{k-1}(Q))\leq \cdots\leq
\delta^k\,\omega_1(Q).$$

Now let $k$ be the minimal natural number such that $\delta^k\leq\ve/10$.
Denote 
$$E_0 = Q\setminus \LD^k(Q),$$
and observe {$\omega_1(E_0)\geq (1-\frac1{10}\ve)\,\omega_1(Q)$} and that 
\begin{equation}\label{eqdre10}
\Theta^n_{\omega_1}(B(x,r))\gtrsim \delta^k\,\Theta^n_{\omega_1}(Q)\quad \mbox{for $\omega_1$-a.e.\ $x\in E_0$ and all $r\in(0,\ell(Q))$.}
\end{equation}

{We claim now that we may find another
subset $E_1\subset Q$  satisfying
$\omega_1(E_1)\geq (1-\tfrac1{10}\ve)\,\omega_1(Q)$ and such that
\begin{equation}\label{eqdre11}
\Theta^n_{\omega_1}(B(x,r))\lesssim C(\ve)\,\Theta^n_{\omega_1}(Q)\quad \mbox{for all  $x\in E_1$ and all $r\in(0,\ell(Q))$.}
\end{equation}
Indeed, according to Lemma \ref{l:goodset}, for any $\ve'>0$, there exists a set $\wt E_1\subset B_Q$ such that
$\omega_1(B_Q\setminus \wt E_1) \leq \ve'\omega_1(B_Q)$ and 
\begin{equation}\label{eqdre121}
\frac{\omega_1(B(x,r))}{\omega_2(B(x,r))}\approx_{\ve'} \frac{\omega_1(B_Q)}{\omega_2(B_Q)} \approx
\frac{\omega_1(Q)}{\omega_2(Q)}\quad \mbox{ for all $x\in \wt E_1$, $0<r< \ell(Q)$.}
\end{equation}
Let $E_1= \wt E_1\cap Q$ and notice that 
$$\omega_1(Q\setminus E_1) \leq \omega_1(B_Q\setminus \wt E_1)\leq \ve'\omega_1(B_Q)\leq C\ve'\omega_1(Q),$$
so that, for $\ve'$ small enough, $\omega_1(Q\setminus E_1)\leq \frac1{10}\ve\,\omega_1(Q)$.
Also, by Lemma \ref{l:otherside}, \rf{eqdre121}, Theorem \ref{t:ACF}, and the doubling property of $\omega_1$ and $\omega_2$,
for  $x\in  E_1$ and $0<r< \ell(Q)$,
\begin{align*}
\Theta^n_{\omega_1}(B(x,r)) &= \big(\Theta^n_{\omega_1}(B(x,r))\, \Theta^n_{\omega_2}(B(x,r))\big)^{1/2}\,\left(\frac{\Theta^n_{\omega_1}(B(x,r))}{\Theta^n_{\omega_2}(B(x,r))}\right)^{1/2}\\
& \lesssim_{\ve}  J(x,2r)^{1/4}\,\left(\frac{\Theta^n_{\omega_1}(Q)}{\Theta^n_{\omega_2}(Q)}\right)^{1/2}
\leq J(x,2\ell(Q))^{1/4}\,\left(\frac{\Theta^n_{\omega_1}(Q)}{\Theta^n_{\omega_2}(Q)}\right)^{1/2}\\
& \lesssim \big(\Theta^n_{\omega_1}(B(x,2C_1\ell(Q)))\, \Theta^n_{\omega_2}(B(x,2C_1\ell(Q)))\big)^{1/2}\,\left(\frac{\Theta^n_{\omega_1}(Q)}{\Theta^n_{\omega_2}(Q)}\right)^{1/2} \\
&\lesssim \Theta^n_{\omega_1}(Q).
\end{align*}
So \rf{eqdre11} holds, and the claim above is proven.}

Next, let $E' = E_0\cap E_1$, so that both \rf{eqdre10} and \rf{eqdre11} hold $\omega_1$-a.e.\ in $E'$ and
moreover
{
\begin{align*}
\omega_1(Q\setminus E') & = \omega_1(Q\setminus ( E_0\cap E_1)) \leq 
 \omega_1(Q\setminus E_0) + \omega_1(Q\setminus E_1)\\
 & \leq \tfrac1{10}\ve\,\omega_1(Q) + 
 \tfrac1{10}\ve\,\omega_1(Q) = \tfrac15\ve\,\omega_1(Q).
\end{align*}
}
Consider the measure
$$\nu = \frac1{\Theta^n_{\omega_1}(Q)}\,\omega_1|_{E'}.$$
Arguing as in Lemmas \ref{lem:applic-Tb-suppressed} and \ref{lem:G-estimates} we deduce that
the Riesz transform $\RR_\nu$ is bounded in $L^2(\nu)$ with norm possibly
depending on $\ve$. Also, from Theorem \ref{t:pajot}, Lemma \ref{lem:nu-BP-ADR}, and Remark \ref{rempq},
it follows that there exists an $n$-AD-regular measure $\sigma$ and a subset $G\subset E'$ such that
$\nu(G)\geq(1-\frac15\ve)\nu(E')$, $\sigma$ coincides with $\nu$ on $G$, and
$\RR_\sigma$ is bounded in $L^2(\sigma)$, with its norm depending on $\ve$, or equivalently, $\sigma$ 
is uniformly rectifiable. It is immediate to check that the sets $E=\supp\sigma$ and $G$ satisfy the 
required properties \rf{equr00*}, \rf{equr01*}.

\vv
\noi (b) $\Rightarrow$ (c): Trivial.

{
\vv

\noi (c) $\Rightarrow$ (d):
Suppose that $\omega_1$ has big pieces of uniformly $n$-rectifiable measures. Consider a ball $B$
centered at $\partial\Omega_1$. Let $E$ and $G$ be as in the definition of big pieces of uniformly $n$-rectifiable measures for $\omega_1$, so that
$\omega_1(G)\approx \omega_1(B),$
 which implies that $\HH^n(G)\approx 
r(B)^n$, because $\omega_1$ is comparable to $\Theta^n_{\omega_1}(Q)\HH^n$ on $G$.

From \cite[Theorem 6.4]{Azzam} we deduce that there exist two chord-arc domains $\Omega_{B,i}\subset\Omega_i$
such that $G\subset\partial\Omega_i \cap \partial \Omega_{B,i}$, as wished.

\vv

\noi (d) $\Rightarrow$ (a): 
 Let $E\subset B\cap\partial\Omega_1$ satisfy
$\omega_1(E)\geq \delta\,\omega_1(B)$ for some $\delta\in (0,1)$. We intend to check that
if $\delta$ is close enough to $1$, then $\omega_2(E)\geq \ve\,\omega_2(B)$, for some $\ve>0$ not depending on $E$.

Let $\Omega_{B,i}$, for $i=1,2$, be as in the statement (d), and let $G=\partial\Omega_{B,1}\cap 
\partial\Omega_{B,2}\subset B\cap \partial\Omega$.
Let $p_{i}\in\Omega_{B,i}$ be such that
$\dist(p_i,\partial\Omega_{B,i})\approx \dist(p_i,G)\approx r(B)$.
By the change of pole formula, we have
$$\frac{\omega_1^{p_1}(B\setminus E)}{\omega_1^{p_1}(B)}\approx 
\frac{\omega_1(B\setminus E)}{\omega_1(B)} \leq 1-\delta.$$
Hence, by the maximum principle,
$$\omega_{\Omega_{B,1}}^{p_1}(G\setminus E)\leq
\omega_1^{p_1}(G\setminus E)\leq \omega_1^{p_1}(B\setminus E)\leq C(1-\delta).$$
Notice that $\omega_{\Omega_{B_1}}^{p_1}(G)\approx1$, because
$\HH^n(G)\gtrsim r(B)^n$ and $\Omega_{B,1}$ is a chord-arc domain. 
Thus, for $\delta$ close enough to $1$, 
$$\omega_{\Omega_{B,1}}^{p_1}(G\setminus E)\leq \frac12\,\omega_{\Omega_{B,1}}^{p_1}(G),$$
or equivalently, $\omega_{\Omega_{B,1}}^{p_1}(E\cap G)\geq \frac12\,\omega_{\Omega_{B,1}}^{p_1}(G)\approx 1$.
Hence, $\HH^n(E\cap G)\gtrsim r(B)^n$ (using again that ${\Omega_{B,1}}$ is a
 chord-arc domain), and this in turn implies that 
 $\omega_{\Omega_{B,2}}^{p_2}(E\cap G)\gtrsim 1$ (using now that ${\Omega_{B,2}}$ is another
 chord-arc domain). Then, by the maximum principle, 
 $$\omega_2^{p_2}(E\cap G)\geq \omega_{\Omega_{B,2}}^{p_2}(E\cap G)\gtrsim 1.$$ 
 By the change of pole formula, this implies that
$$\omega_2(E)\geq \omega_2(E\cap G)\gtrsim \omega_2(B),$$
as wished.
}

\end{proof}

\vv

\section{The Example}\label{sec:7}

In this section we prove Proposition \ref{p:example}. 
First we will introduce the precise domains $\Omega_1,\Omega_2$ for the proof of the proposition.

We construct a curve $\Gamma$ similar to the von Koch curve, but with a small change. Let $\Gamma_{0}$ be the unit line segment. Now let $\Gamma_{1}$ be the four segments in the first stage of the von Koch construction. Then we freeze the third edge and perform no alterations on this segment, and then replace the other three segments with scaled copies of $\Gamma_{1}$. The resulting curve is $\Gamma_{2}$. We freeze the corresponding third edges of each of these three copies, and then replace the remaining edges with a scaled copy of $\Gamma_{1}$ to obtain $\Gamma_{3}$, and so on. We set $\Gamma$ to be the Hausdorff limit of these curves. See Figure \ref{f:koch}.

\begin{figure}[!ht]
\includegraphics[width=400pt]{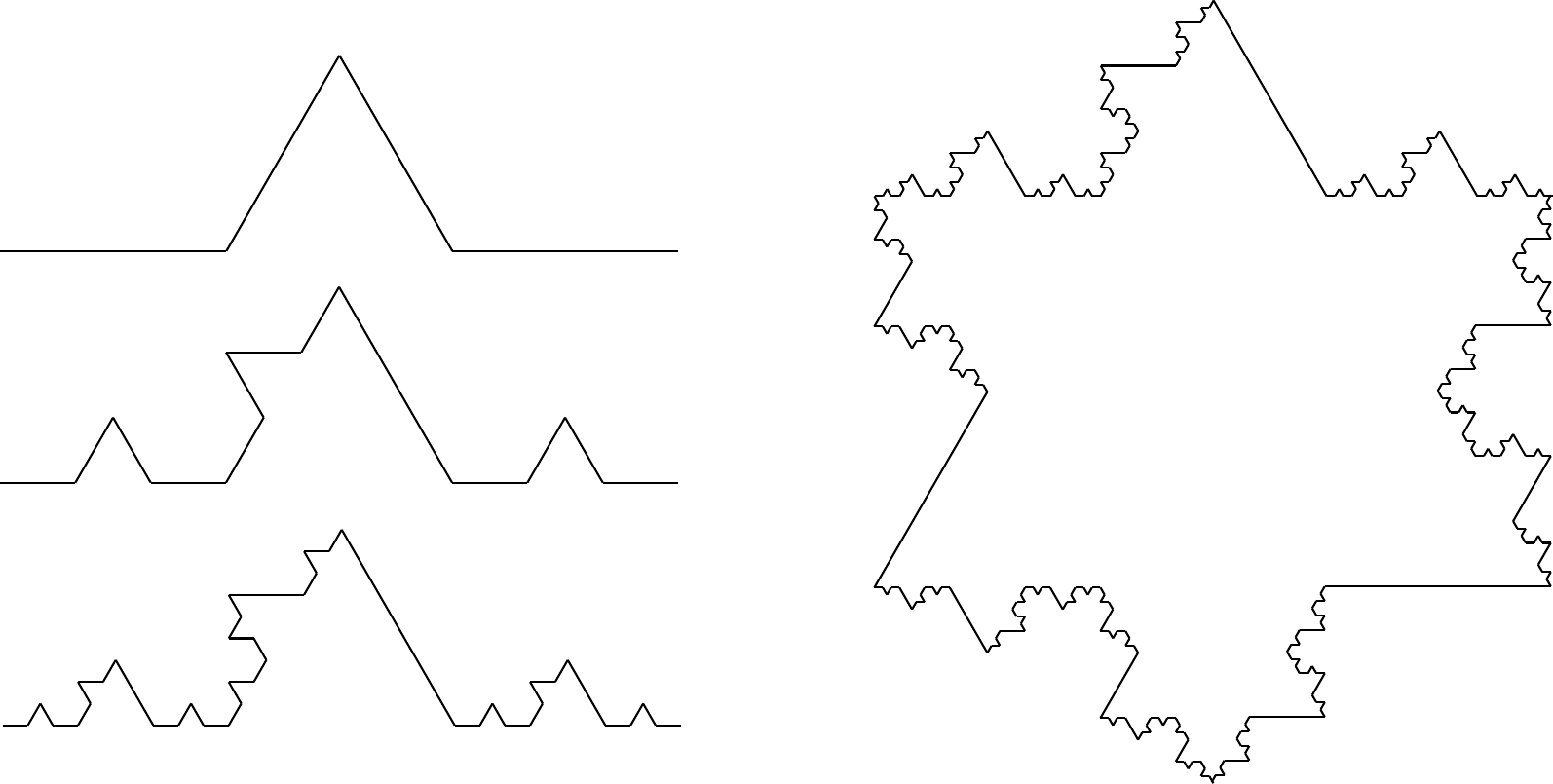}

\begin{picture}(400,0)(0,0)
\put(25,155){1}
\put(60,175){2}
\put(105,175){3}
\put(140,155){4}
\put(305,110){$\Omega_1$}
\put(225,110){$\Sigma$}
\end{picture}
\caption{On the left, we have the first three stages of the construction of $\Gamma$. The first curve is $\Gamma_{1}$, and we freeze edge number 3. Then replace edges 1,2, and 4 with copies of $\Gamma_{1}$ to make $\Gamma_{2}$, then replace each new edge there (except for the images of edge number 3) with copies of $\Gamma_{1}$ to get $\Gamma_{3}$, and so on. On the right, we take 3 copies of $\Gamma$ to obtain a quasicurve $\Sigma$ that splits the plane into two NTA domains. }
\label{f:koch}
\end{figure}
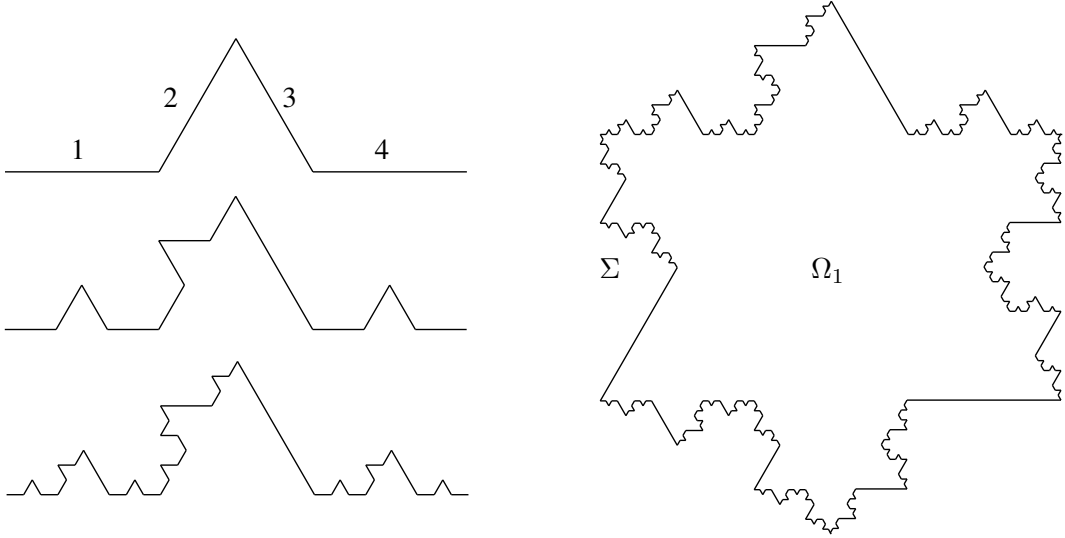

More explicitly, we can define $\Gamma$ as follows. Let $\Gamma_{0}$ be the unit line segment and let $f_{1},f_{2},f_{3},f_{4}$ be the similarities that map $\Gamma_{0}$ onto four line segments of 1/3 the size as in the usual von Koch construction. Now we let $K$ be the attractor of the similarities $f_{1},f_{2},$ and $f_{4}$ (that is, we omit $f_{3}$), so
\[
K=f_{1}(K)\cup f_{2}(K)\cup f_{4}(K). 
\]
We now let 
\[
\cL_{n} = \{f_{3}\circ f_{\alpha} (\Gamma_{0}): |\alpha|=n, \alpha_{i}\neq 3 \mbox{ for any }i\}
\]
and
\[
\Gamma = K\cup J,\;\; J:=\bigcup_{n\geq 0} \bigcup_{L\in \cL_{n}}L,
\]
where the union is over all multiindices $\alpha=\alpha_{n}\alpha_{n-1}\cdots \alpha_{1}$ (including the empty index) and $f_{\alpha} = f_{\alpha_{n}}\circ \cdots \circ f_{\alpha_{1}}$ (where $f_{\varnothing}$ is just the identity).  Now combine three rotated copies (by angles of $0$, $\frac{2\pi}{3}$, and $\frac{4\pi}{3}$ respectively) to obtain a closed curve $\Sigma$ that splits the plane into two open sets $\Omega_1$ and $\Omega_2$ as in Figure \ref{f:koch}. Using arguments similar to the original von Koch snowflake, is not hard to check that
\begin{enumerate}
\item $\Sigma$ is a quasicircle (see  \cite[Chapter VI, Equation 2.5]{GM}). 
\item $\Sigma$ consists of a purely unrectifiable set of finite and positive length (which is the three rotated copies of $K$) and a countable union $J$ of line segments with $\HH^{1}(J)=\infty$ (see \cite{Mattila} for a reference on rectifiable sets), and 
\item for every ball $B$ centered on $\Sigma$ with $r_{B}<\diam \Sigma$, there is another ball $B'\subset B$ centered also at $\Sigma$ with radius comparable to $r_{B}$ such that $\Sigma\cap B'$ is a 
 line segment $L_{B'}$, and one component of $B'\setminus L_{B'}$ is contained in $\Omega_1$ and the other in $\Omega_2$.
\end{enumerate}

Since $\Sigma$ is a closed quasicurve, it splits the complex plane into two NTA domains $\Omega_{1}$ and $\Omega_{2}$ (use \cite[Theorem VI.3.3]{GM} and the fact that quasiconformal mappings send NTA domains to NTA domains). 
{
Further, from the property (3), it is clear that $\Omega_1$ and $\Omega_2$ have joint big pieces
of chord-arc subdomains, and thus by Theorem \ref{teoNTA}, $\omega_2\in A_\infty(\omega_1)$, which
concludes the proof of Proposition \ref{p:example}.}
\vv

\begin{rem}
Notice that in our construction, we could have skipped a few more generations before freezing an edge, and so we could certainly construct domains whose boundaries have dimension strictly larger than one where interior and exterior harmonic measures are $A_{\infty}$-equivalent (and by choosing constructions other than variants of the von Koch curve, we can make the dimension as close to 2 as we wish). 
\end{rem}




\end{document}